\documentclass[12pt]{article}  

\usepackage{amssymb}

\usepackage[dvipsnames]{xcolor}					
\usepackage{color}								
\usepackage[toc,title]{appendix}				

\usepackage{comment}							
\usepackage[numbers]{natbib}					
\usepackage{pdfpages}                           
\usepackage{xspace}								
\usepackage{epstopdf} 							
\usepackage{nomencl}							
\usepackage{etoolbox}							
\usepackage{longtable}							
\usepackage{authblk} 

\usepackage[ margin=2cm]{geometry}

\renewcommand\nomgroup[1]{%
	\item[\bfseries
	\ifstrequal{#1}{M}{Matrices}{%
		\ifstrequal{#1}{V}{Vectors}{%
			\ifstrequal{#1}{Sets}{Sets}{}}}%
	]}

\usepackage{graphicx}							
\usepackage{multirow}							
\usepackage{paralist}							
\usepackage{verbatim}							
\usepackage{longtable}							
\usepackage{booktabs}							
\usepackage{subcaption}							

\makeatletter
\def\BState{\State\hskip-\ALG@thistlm}
\makeatother

\usepackage{tikz}								
\usetikzlibrary{graphs,graphs.standard,quotes}	
\usetikzlibrary{backgrounds}					
\usetikzlibrary{fit} 							

\usepackage[T1]{fontenc}
\usepackage{lmodern}
\usepackage[utf8]{inputenc}						
\usepackage{soul}								
\usepackage[normalem]{ulem}						
\usepackage{mathrsfs}							

\usepackage[linesnumbered,noend]{algorithm2e}   
\usepackage{setspace}
\DontPrintSemicolon
\SetKw{KwAnd}{and}
\SetKw{KwOr}{or}
\SetKw{KwTrue}{true}
\SetKw{KwFalse}{false}

\usepackage{forloop}							
\usepackage{ifthen}								

\colorlet{darkred}{red!80!black}
\colorlet{darkgreen}{green!60!black}
\colorlet{darkblue}{blue!80!black}
\colorlet{darkorange}{orange!70!black}
\definecolor{Green}{cmyk}{1, 0.2, 0.4, 0.1}
\definecolor{Orange}{cmyk}{0, 0.61, 0.87, 0}
\definecolor{Purple}{rgb}{0.75, 0.0, 1.0}
\definecolor{purple}{rgb}{0.5,0,0.5}
\definecolor{dgreen}{rgb}{0.1,0.9,0.5}
\definecolor{gabysgreen}{cmyk}{0.80, 0.1, 0.90, 0}
\newcommand{\red}[1]{{\color{black}#1}}
\newcommand{\redd}[1]{{\color{black}#1}}
\newcommand{\blue}[1]{{\color{black}#1}}

\usepackage{amssymb}							
\usepackage[fleqn]{amsmath}						
\usepackage{amsthm}								
\usepackage{bm}									
\usepackage{sansmath}
\usepackage{mathtools}							
\usepackage{nicefrac}							
\usepackage{siunitx}							
\usepackage{xfrac}								

\usepackage{hyperref}							
\hypersetup{
	colorlinks,
	linkcolor={blue!50!black},
	citecolor={green!50!black},
	urlcolor={blue!80!black}
}
\usepackage[capitalise,noabbrev]{cleveref}		

\newcommand{\irg}[2]{[#1\!:\!#2]}
\newcommand{\R}{\mathbb{R}}						
\newcommand{\N}{\mathbb{N}}						
\newcommand{\CPP}{\mathcal{CPP}}				

\newcommand{\CCP}{\mathcal{CCP}}
\newcommand{\PCP}{\mathcal{PCP}}

\newcommand{\x}{\vc{x}}							
\renewcommand{\u}{\vc{u}}						
\newcommand{\lrbr}[1]{\left\lbrace #1 \right\rbrace} 

\newcommand\ga{\alpha}

\newcommand\gl{\lambda}
\newcommand\ggl{\bm{\lambda}}

\def\x{\mathbf x}
\def\aa{\mathbf a}
\def\y{\mathbf y}
\def\z{\mathbf z}

\def\v{\mathbf v}
\def\w{\mathbf w}

\def\f{\mathbf f}
\def\oo{\mathbf o}

\def\u{\mathbf u}
\def\b{\mathbf b}
\def\cc{\mathbf c}
\def\d{\mathbf d}

\def\g{\mathbf g}

\def\q{{\mathbf q}}

\def\AA{{\mathcal A}}

\newcommand\CC{{\mathcal C}}

\newcommand\EE{{\mathcal E}}
\newcommand\FF{{\mathcal F}}
\newcommand\GG{{\mathcal G}}
\newcommand\HH{{\mathcal H}}
\newcommand\II{{\mathcal I}}

\newcommand\KK{{\mathcal K}}

\newcommand\NN{{\mathcal N}}

\newcommand\PP{{\mathcal P}}
\newcommand\RR{{\mathcal R}}

\def\SS{{\mathcal S}}

\newcommand\VV{{\mathcal V}}

\newcommand\XX{{\mathcal X}}

\newcommand\YY{{\mathcal Y}}

\newcommand\Ab{{\mathsf{A}}}
\newcommand\Bb{{\mathsf B}}
\newcommand\Cb{{\mathsf C}}

\newcommand\Fb{{\mathsf F}}

\newcommand\Ib{{\mathsf I}}

\newcommand\Mb{{\mathsf M}}

\newcommand\Ob{{\mathsf O}}

\newcommand\Qb{{\mathsf Q}}

\newcommand\Vb{{\mathsf V}}

\newcommand\Xb{{\mathsf X}}
\newcommand\Yb{{\mathsf Y}}
\newcommand\Zb{{\mathsf Z}}

\newcommand{\T}{^\mathsf{T}}	
\DeclareMathOperator{\conv}{conv}									
\DeclareMathOperator{\diag}{diag}									
\DeclareMathOperator{\Diag}{Diag}									
\def\beq#1{$$}														

\def\bea#1{\begin{array}{#1}}
	\def\ea{\end{array}}
\def\ignore#1{}

\theoremstyle{plain}
\newtheorem{thm}{Theorem} 
\newtheorem{exmp}{Example} 

\newtheorem{lem}[thm]{Lemma}
\newtheorem{prop}[thm]{Proposition}
\newtheorem*{rem}{Remark}
\theoremstyle{definition}

\theoremstyle{remark}

\usepackage{fancyhdr}

\begin{document}
	
	\title{Conic optimization techniques yield sufficient conditions for set-completely positive matrix completion under arrowhead specification pattern \thanks{Supported by the
			Austrian Science Fund (project ESP 486-N).}}
	
	\author{ Markus Gabl} 
	\affil{Faculty of Mathematics, University of Vienna, Austria. \texttt{markus.gabl@univie.ac.at}\\
	}
	\renewcommand\Authands{ and }
	\maketitle
	
	\begin{abstract}
		Matrix completion results deal with the question of when a partially specified symmetric matrix can be completed to a member of certain matrix cones. Results from positive semidefinite matrix completion and completely positive matrix completion have been successfully applied in optimization to greatly reduce the number of variables in conic optimization problems in the space of symmetric matrices. In this text, we go the other direction and show that we can use tools from conic optimization (more precisely: from copositive optimization) to establish a new completion result that complements the existing literature in two regards: firstly, we consider set-completely positive matrix completion, which generalizes completion with respect to the traditional completely positive matrix cone. Secondly, we consider a specification pattern that is not in the scope of classical results for completely positive matrix completion. Namely, we consider arrow-head specification patterns where the \red{width is} equal to one. \blue{Our theory is applied to a class of quadratic optimization problems.} 
	\end{abstract}
	
	

	\pagebreak	
	
	\section{Introduction}\label{chap:ReducingtheSizeofConicReformulations} 
	
	This text is concerned with the theory of set-completely positive matrix completion. Given a matrix, whose entries are partially unspecified, the question is whether there are values for the unspecified entries such that the matrix can be completed to a member of set-completely positive matrix cones, such as the positive semidefinite matrix cone or the completely positive matrix cone.  
	
	This branch of linear algebra has been hugely influential in the field of sparse reformulations of conic optimization problems over the space of symmetric matrices, which are optimization problems that are linear except for a single, nonlinear, conic constraint. Such problems often arise from convex reformulations and relaxations of nonconvex optimization problems. Results involving the positive semidefinite matrix are plentiful and include \cite{goemans_improved_1995,shor_quadratic_1987,yakubovich_s-procedure_1971} to name a few; a more exhaustive (but by no means complete) survey is given in \cite{vandenberghe_semidefinite_1996}. Another strain of literature is concerned with reformulations based on the completely positive matrix cone (which contains all matrices of the form $\Bb\Bb\T$, where $\Bb$ is nonnegative) and its generalizations (see \cite{bomze_copositive_2000,burer_copositive_2009,eichfelder_set-semidefinite_2013} as well as the surveys \cite{bomze_optimization_2023,bomze_think_2012,dur_copositive_2010}). In many interesting cases, the linear constraints and objective of these problems exhibit sparsity patterns, i.e. patterns of zeros, so that certain entries of the matrix variable appear only in the conic constraint. 
	
	In such cases, sparse relaxations allow replacing the computationally costly conic constraint by a collection of such constraints on certain submatrices of the matrix variable (ideally of small order), thereby eliminating spurious entries of the matrix variable from the problem description entirely. Such reformulations can provide substantial computational benefits and have therefore been studied extensively.
	
	Results regarding the positive semidefinite matrix cone were first studied in \cite{fukuda_exploiting_2001}, which sparked a plethora of investigations, an excellent survey of which is given in \cite{vandenberghe_chordal_2015}. 
	Reformulations involving the completely positive matrix cone were discussed for example in \cite{kim_doubly_2020,gabl_sparse_2023}.
	To ensure that a sparse relaxation is exact, one typically employs matrix completion results, most importantly \cite{grone_positive_1984} for psd \red{(positive semi-definite)} completion and \cite{drew_completely_1998} for completely positive completion. Suppose the submatrices present in the sparse relaxation can be completed to a matrix that fulfills the original conic constraint. In that case, the relaxation gap can be closed and the computational benefits can be reaped at no additional cost. 
	
	Unfortunately, said results on matrix completion impose a lot of restrictions on the sparsity pattern and the type of conic constraint, which we will discuss in greater detail later in the text. In this note, we want to improve upon existing results in this regard. To this end, we will consider sparsity patterns with a certain type of arrowhead structure. Matrices with this sparsity pattern are not within the domain of classical matrix completion results as far as the completely positive matrices are concerned. Our main result is a new sufficient condition for completability of matrices with arrow-head specification pattern, where the \red{width of the diagonal is equal to one}, and where the completion is with respect to certain set-completely positive matrix cones that are more general than the 'traditional' completely positive matrix cone. Set-completely positive completion has been considered before in \cite{gabl_sparse_2023}, but results on sufficient conditions therein were very limited. We arrive at these results by employing copositive optimization techniques. This is interesting because, rather than using linear algebra results for the benefit of optimization, we use techniques from optimization in order to tackle a problem in linear algebra. However, optimization also benefits from our result since it can potentially be used to strengthen conservative approximations of sparse reformulations discussed in \cite{gabl_sparse_2023}, but we also use it \redd{in} the present text to derive new ex-post certificates of exactness for sparse relaxations of a certain type of completely positive optimization problem\redd{, i.e.,} certificates deduced \redd{from the optimal solution of the relaxation} rather than from the problem structure. We hope that our approach inspires further investigation into the subject. Before laying out our derivation in greater detail, we will briefly review some important concepts regarding conic reformulations, sparse relaxations, and matrix completion and introduce some notation and preliminaries.   
	
	\subsection{Notation \red{and preliminaries}}
	Throughout the paper, matrices are denoted by sans-serif capital letters (e.g.\ $\Ob$ will denote the zero matrix, where the size will be clear from the context), vectors by boldface lowercase letters (e.g.\ $\oo$ will denote the zero vector), and scalars (real numbers) by simple lowercase letters. Sets will be denoted using calligraphic letters, as will be a certain function but context will clarify. For example, cones will often be denoted by $\KK$. We use $\SS^n$ to indicate the set of $n\times n$ symmetric matrices and $\SS_+^n/\SS_-^n$ for the sets of $n\times n$ positive-/negative-semidefinite matrices, respectively. Moreover, we use $\NN_n$ to denote the set of $n\times n$ entrywise nonnegative, symmetric matrices. Exceptions to this rule will be the index set $\irg{l}{k}\coloneqq \lrbr{l,l+1,\dots,k-1,k}\subseteq \N$, further $\R^n$, the set of real $n$-dimensional vectors, and its subset of nonnegative vectors denoted by $\R^n_+$. For a given set $\AA$ we denote its closure, interior, \red{relative interior} and convex hull by $\mathrm{cl}(\AA)$, $\mathrm{int}(\AA)$, \red{$\mathrm{ri}(\AA)$,}  $\mathrm{conv}(\AA)$ respectively. 
	
	For a cone $\KK$ in some vector space $\VV$ we denote the dual cone by $\KK^*\coloneqq \lrbr{\x\in\VV\colon \langle \y, \x \rangle\geq 0 , \ \forall \y\in\KK}$. We will several times use the facts that $\KK^{**} = \mathrm{cl}\red{(}\mathrm{conv}(\KK)\red{)}$ and $\mathrm{int}\redd{(\KK^*)} = \lrbr{\red{\x} \colon \langle\red{\y,\x}\rangle >0,\ \forall \red{\y}\in\KK\setminus\lrbr{\oo}}$ for any cone $\KK$ in a \redd{finite-dimensional} vector space. We also make use of the Frobenius product of two appropriately sized matrices $\Ab$ and $\Bb$ defined as $\Ab\bullet\Bb \coloneqq \mathrm{trace}(\Ab\T\Bb)= \sum_{i=1}^n\sum_{j=1}^n(\Ab)_{ij}(\Bb)_{ij}$. 
	
	At several points in the text, we will invoke the concept of the recession cone of a nonempty and convex set $\CC\subseteq \R^n$, which is defined as:
	\begin{align}
		\CC^{\infty}
		\coloneqq 
		\lrbr{\bar\x\in\R^n
			\colon 
			\x+\gl\bar\x \in\CC, \ \forall (\x,\gl)\in\CC\times\R_+ 
		},
	\end{align}
	\blue{
	The following proposition summarizes properties of recession cones relevant to our exposition. 
	\begin{prop}\label{prop:ReccCones}
		The following statements on recession cones hold: 
		\begin{enumerate}
			\item[a)] A nonempty closed convex set $\CC\subseteq \R^n$ is bounded if and only if $\CC^{\infty} = \lrbr{\oo}$.
			
			\item[b)] For a closed convex cone $\KK$ we have $\KK^{\infty} = \KK$.
			
			\item[c)] For a hyperplane $\HH\coloneqq \lrbr{\x\in\R^n\colon \aa\T\x = b}$ we have $\HH^{\infty} = \lrbr{\x\in\R^n\colon \aa\T\x = 0}$.
			
			\item[d)] For two nonempty closed convex sets $\CC_i\subseteq \R^n, \ i = 1,2$ with nonempty intersection we have $(\CC_1\cap\CC_2)^{\infty} = \CC_1^{\infty}\cap\CC_2^{\infty}$.
			
			\item[e)] In particular, for $\CC \coloneqq \lrbr{\x\in\KK \colon \Ab\x = \b}\neq \emptyset$, where $\KK\subseteq \R^n$ is a closed convex cone, we have  $\CC^{\infty} = \lrbr{\x\in\KK \colon \Ab\x = \oo}$.

		\end{enumerate}
	\end{prop}
	\begin{proof}
		The first and the fourth statements are proved in \cite[Theorem 8.4., Corollary 8.3.3.]{rockafellar_convex_2015} respectively. The second and third statement are proved by elementary arguments. The final statement follows from the three preceding ones.  
	\end{proof}
	}
	
	\subsection{Conic and copositive optimization}
	By conic optimization problems over the space of symmetric matrices, we mean problems of the form  
	\begin{align}\label{eqn:ConicProblem}
		\min_{\Xb\in\SS^n}
		\lrbr{
			\Cb\bullet \Xb \colon \AA(\Xb) = \b, \ \Xb\in\CC
		},
	\end{align} 
	where $\AA \colon \SS^n \rightarrow \R^m$ is a linear map, the vector $\b\in\R^m$, and $\CC\subseteq\SS^n$ is a nonempty, closed, convex matrix cone. We will be concerned with the case where the latter cone is given by the set-completely positive matrix cone, which is defined as
	\begin{align*}
		\CPP(\KK)\coloneqq \conv\lrbr{\x\x\T\colon \x\in\KK} = \lrbr{\Xb\Xb\T\colon \Xb\in\KK^r,\ r\in\N}\subseteq \SS^n_+,
	\end{align*}  
	for a nonempty closed convex cone  $\KK\subseteq\R^n$ called the ground cone and \red{$\KK^r\subseteq \R^{n\times r}$ is the cone of $n\times r$ matrices whose columns are vectors in $\KK$}. For example, $\CPP(\R^n) =\SS^n_+$, and  $\CPP(\R^n_+)$ is the 'traditional' completely positive matrix cone introduced in \cite{motzkin_copositive_1818} and extensively studied in \cite{berman_completely_2003,dickinson_copositive_2013}. Since they are subsets of the psd cone, the $\CPP$ cones are \redd{pointed, i.e.,} they do not contain a line. Also, by \cite[Proposition 11]{bomze_optimization_2023}, they are closed if the respective ground cone is closed. \red{All ground cones in this article are henceforth assumed to be nonempty, closed, and convex, so that respective set-completely positive matrix cones are nonempty and closed as well.}
	Another  useful fact is that for $\KK_1\subseteq \R^{n_1},\ \KK_2\subseteq \R^{n_2}$ and $\Xb\in\SS^{n_1},\ \Yb\in\SS^{n_2},\ \Zb\in\R^{n_2\times n_1}$ we have that
	\begin{align}\label{eqn:SubconesOfCPP}
		\begin{bmatrix}
			\Xb & \Zb\T \\ \Zb & \Yb 
		\end{bmatrix}\in\CPP(\KK_1\times\KK_2)
		\mbox{ implies } \Xb\in\CPP(\KK_1), \ \Yb\in\CPP(\KK_2),
	\end{align}
	which is immediate from the definition of set-completely positive matrix cones. 
	Throughout the paper, we will write $\CPP$ instead of $\CPP(\KK)$, whenever the ground cone is not important to the discussion or clear from the context, and we do the same with similar cones discussed shortly. 
	
	Note, that if in the description of the feasible set of \eqref{eqn:ConicProblem} we replace $\Xb\in\CC$ by $\Xb= \x\x\T,\ \x\in\KK$, then it represents a generic quadratically constrained conic quadratic problem (QCCQP). Indeed, obtaining relaxations for this difficult class of problems is one of the major motivations for studying \eqref{eqn:ConicProblem} with $\CC= \CPP$. A central object in this regard is the set: 
	\begin{align*}
		\GG(\FF) 
		\coloneqq
		\mathrm{cl}\conv
		\lrbr{
			\begin{bmatrix}
				1 \\ \x
			\end{bmatrix}
			\begin{bmatrix}
				1 \\ \x
			\end{bmatrix}\T
			\colon \x\in\FF
		}\subseteq \SS^{n+1}, 
		\mbox{ where } \FF\subseteq \R^n \mbox{ (an arbitrary feasible set)},
	\end{align*}
	and its set-completely positive characterizations, which feature prominently in the literature on copositive reformulations of nonconvex QCCQPs (see e.g. \cite{burer_copositive_2009,burer_copositive_2012,eichfelder_set-semidefinite_2013}). For a quadratic optimization problem over an arbitrary set $\FF\subseteq \R^n$ we always have:
	\begin{align}\label{eqn:chlifting}
		\hspace{-0.6cm}\inf_{\x\in \FF} 
		\lrbr{
			\x\T \Qb\x +2 \q\T\x
		} 
		= 
		\inf_{
			\substack{\x\in\R^n,\\ \Xb\in\SS^n}
		}
		\lrbr{
			\Qb\bullet\Xb+ 2\q\T\x
			\colon 
			\begin{bmatrix}
				1 & \x\T \\ \x & \Xb 
			\end{bmatrix} 
			\in \GG(\FF)
		}\, ,
	\end{align}
	which is proved for example in \cite[Theorem 1]{bomze_optimization_2023} but earlier treatments are found in \cite{anstreicher_convex_2012,burer_copositive_2009,eichfelder_set-semidefinite_2013}.
	Most importantly, we will make use of the following characterization derived in \cite[Theorem 1.]{burer_copositive_2012}:
	\begin{align}\label{eqn:characterizationofGGFF}
		\hspace{-1 cm}
		\GG(\FF) = 
		\lrbr{
			\begin{bmatrix}
				1 & \x\T \\ \x & \Xb
			\end{bmatrix}	
			\in\CPP(\R_+\times\KK)
			\colon 
			\begin{array}{rl}
				\Ab\x =& \hspace{-0.2cm} \b,\\
				\Diag\left(\Ab\Xb\Ab\T\right)=& \hspace{-0.2cm} \b \circ \b
			\end{array}			
		}, \mbox{ for }
		\FF \coloneqq \lrbr{\x\in\KK\colon \Ab\x= \b},
	\end{align}
	where the \red{ground cone} $\KK\subseteq\R^n$ and the \redd{$\circ$-operation gives the vector} of elementwise products. Note that the right-hand side problem in \eqref{eqn:chlifting} with $\GG(\FF)$ as in \eqref{eqn:characterizationofGGFF} is indeed an instance of \eqref{eqn:ConicProblem}, as the objective and the constraints are linear except for the completely positive matrix constraint. \blue{We also want to highlight the fact that a matrix in $\GG(\FF)$ (given as in \eqref{eqn:characterizationofGGFF}) has an $\x$-component that is a member of $\FF$, which for ease of reference is argued in the following proposition. 
		
		\begin{prop} \label{prop:LinearPartinKK}
			Let $\FF\subseteq \R^n$ and $\GG(\FF)\subseteq\SS^{n+1}$ be given as in \eqref{eqn:characterizationofGGFF} and let $\x\in\R^n$ be the respective northeast (or southwest) component of a matrix in $\GG(\FF)$. Then $\x\in\FF\subseteq\KK$.  
		\end{prop}
		
		\begin{proof}	
			From the containment of the respective matrix in $\CPP(\R_+\times\KK)$ and the definition of that cone, we get
			\begin{align*}
				\begin{bmatrix}
					1 & \x\T \\ \x & \Xb
				\end{bmatrix}	 
				=
				\sum_{l=1}^{k}
				\begin{bmatrix}
					\mu_l \\ \x_l 
				\end{bmatrix} 
				\begin{bmatrix}
					\mu_l \\ \x_l 
				\end{bmatrix} = 
				\begin{bmatrix}
					\sum_{l=1}^{k} \mu_l^2  & \sum_{l=1}^{k} \mu_l\x_l\T \\
					\sum_{l=1}^{k} \mu_l\x_l  & \sum_{l=1}^{k} \x_l\x_l\T
				\end{bmatrix} ,
				\mbox{ where }  \x_l \in\KK, \ \mu_l\geq 0, \ l\in\irg{1}{k}, 
			\end{align*}
			so that $\x = \sum_{l=1}^{k} \mu_l\x_l\T \in \KK$ since $\KK$ is closed and convex and hence closed under addition and nonnegative scalar multiplication (only positive scalar multiplication would be allowed if the cone did not contain the origin, but all closed cones do). Since $\x$ also fulfills the linear constraint in the description of $\FF$, we conclude that $\x$ is contained \redd{in} the latter set. 
		\end{proof}
		
		This fact will be useful at several points of our discussion. We would like to mention that, while the above proposition rests on the specific form of $\GG(\FF)$ outlined in \eqref{eqn:characterizationofGGFF}, it always holds that the respective component of a matrix in $\GG(\FF)$ has $\x\in\mathrm{clconv}(\FF)$, but no such result is needed in the sequel so we omit the proof. 
	}
	
	\subsection{Sparse counterparts in conic optimization}
	Next, we discuss the sparse counterparts of symmetric matrix cones. For a graph $G\coloneqq \lrbr{\VV,\EE}$ \red{with vertices in $\VV\coloneqq \irg{1}{n}$ and undirected edges in $\EE\subseteq \lrbr{\lrbr{i,j}\subseteq \VV}$} we define a symmetric partial matrix $\Mb_G$ to be a matrix whose entry $(\Mb_G)_{ij}$ is specified if and only if either $\lrbr{i,j}\in\EE$ or $i=j$ and unspecified otherwise. Such entries will be denoted by "$*$". We call $G$ the specification graph of the partial matrix $\Mb_G$. The space of symmetric partial matrices with specifications graph $G$ is denoted as 
	\begin{align*}
		\SS(G) \coloneqq \lrbr{\Mb_G\colon \Mb_G \mbox{ is a partial matrix with specification graph } G},
	\end{align*}
	so that $\SS(G) = \SS^n$ exactly if $G$ is a complete graph. As in this definition, we will \redd{notationally} distinguish elements in $\SS(G)$ from elements of such spaces generated from other graphs, in particular $\SS^n$, by by indexing the partial matrix with the name of the specification graph of the space it belongs to. For example, if $G_1$ and $G_2$ are both specification graphs, then elements of $\SS(G_1)$ will be indexed by $G_1$, e.g. $\Mb_{G_1}\in\SS(G_1)$ and elements of $\SS(G_2)$ will be indexed by $G_2$, e.g. $\Mb_{G_2}\in\SS(G_2)$. This should help readers keep track of what space we operate in at any given point in the text. We will also continue to denote a generic specification graph by $G$.
	
	If for such a partial matrix we can assign values to the unspecified entries such that we obtain a fully specified matrix in $\CPP(\KK)$, we say that the partial matrix is completable with respect to that matrix cone. Accordingly, we define the set
	\begin{align*}
		\CCP_{G}(\KK) \coloneqq \lrbr{\Mb_G\in\SS(G)\colon \Mb_G \mbox{  is completable with respect to } \CPP(\KK)}.
	\end{align*} 
	Note that this set is nonempty, since one can take any element of $\CPP(\KK)$ and declare entries nonspecified according to $G$ in order to obtain an element in $\CCP_{G}(\KK)$.

	This set can now be used to derive sparse reformulations of \eqref{eqn:ConicProblem}. Firstly, note that any linear map $\AA\colon \SS^n\rightarrow \R^m$ can be expressed as 
	\begin{align*}
		\AA(\Xb) = 
		\begin{bmatrix}
			\Ab_1\bullet\Xb \\ 
			\vdots\\
			\Ab_m\bullet\Xb
		\end{bmatrix}
		\mbox{ for some appropriate } \Ab_l\in\SS^n, \ l \in\irg{1}{m}.
	\end{align*}
	Secondly, for two symmetric partial matrices $\lrbr{\Ab_G,\Bb_G}\subseteq \SS(G)$ there is a natural analog of the Frobenius product in the space of partial matrices given by $\Ab_G\bullet\Bb_G = 2\sum_{\lrbr{i,j}\in\EE}(\Ab_G)_{ij}(\Bb_G)_{ij} + \sum_{i=1}^n(\Ab_G)_{ii}(\Bb_G)_{ii}$, i.e.\ the Frobenius product of the two partial matrices after replacing the unspecified entries with zeros. We will not distinguish the two Frobenius products notationally, since it will always be clear from the inputs which version is in order. 
	Now assume that $\AA(.)$ is such that for any  $\lrbr{i,j}\redd{\notin} \EE$ the entries  $(\Ab_1)_{ij}, \dots, (\Ab_m)_{ij}$ are zero so that their Frobenius products with $\Xb$ are constant in $(\Xb)_{ij}$, in other words $\AA(\Xb)$ is independent of those entries. Then we can replace $\Ab_{\redd{l}}, \ \redd{l} \in\irg{1}{m}$ with partial matrices  $\Ab^l_G\in\SS(G), \ l\in\irg{1}{m}$ such that $(\Ab^l_G)_{ij} = (\Ab_l)_{ij}$ if $\lrbr{i,j}\in\EE$ or $i=j$ to obtain a natural \red{restriction} of $\AA(.)$  to the space of symmetric partial matrices, say $\AA_G\colon \SS(G)\rightarrow\R^m$, so that $\AA(\Xb) = \AA_G(\Xb_G)$ whenever $\Xb_G$ is a partial matrix with specification graph $G$ obtained from the (fully specified) matrix $\Xb$ by declaring entries to be unspecified according to $G$. An analogous projection $\Cb_G\in\SS(G)$ of a symmetric matrix $\Cb\in\SS^n$, whose entries with index $\lrbr{i,j}\in\EE$ are again assumed to be zero, provides $\Cb\bullet\Xb = \Cb_G\bullet\Xb_G$. With this in place, it is easy to see that 
	\begin{align}\label{eqn:SparseReformulation}
		\hspace{-0.7cm} 
		\min_{\Xb\in\SS^n}
		\lrbr{
			\Cb\bullet \Xb \colon \AA(\Xb) = \b, \ \Xb\in\CPP(\KK)
		}
		=\hspace{-0.2cm} 
		\min_{\Xb_{G}\in\SS(G)}
		\lrbr{
			\Cb_G\bullet \Xb_{G} 
			\colon 
			\AA_G(\Xb_{G}) = \b, \ 
			\Xb_{\red{G}}\in\CCP_{G}(\KK)
		},
	\end{align}
	since any feasible solution of the latter problem can be completed to an equivalent solution to the former while the converse can be obtained by declaring the respective entries of a feasible solution $\Xb$ to be unspecified.  
	
	The sparse reformulation has fewer variables, but the set $\CCP_{G}(\KK)$ is not workable in general, in the sense that we do not have algorithms to optimize over this set directly. Thus, the standard approach is to replace this matrix cone with an outer approximation given by  
	\begin{align*}
		\PCP_{G}(\KK) \coloneqq 
		\lrbr{
			\Mb_G\in\SS(G)
			\colon 
			\begin{array}{c}
				\mbox{ Every fully specified principal submatrix of $\Mb_{G}$}\\
				\mbox{is the principal submatrix of a matrix in } \CPP(\KK)\\
				\mbox{ induced by the same rows and columns. }
			\end{array}
		}.
	\end{align*} 
	Note that this cone would not change if the condition was instead stated in terms of the maximal fully specified submatrices, i.e.\ fully specified submatrices that are not submatrices of other fully specified submatrices. Thus, we only consider maximal fully specified submatrices henceforth, but refer to those as fully specified submatrices for convenience.    
	
	In many interesting cases, this set can be characterized by using conic constraints on each of the fully specified submatrices. For example, $\PCP_{G}(\R^n)$ is the set of partial matrices, such that every fully specified principal submatrix is positive semidefinite, so that the $\PCP$-relaxation is a \redd{semidefinite optimization problem}, which can be solved, for example, via interior point methods. 
	On the other hand, $\PCP_{G}(\R^n_+)$ contains all the partial matrices where the respective submatrices are completely positive (in the classical sense), so that the $\PCP$-relaxation is a completely positive optimization problem. For these types of problems, practical approaches exist, see for example the respective sections in \cite{burer_copositive_2012,bomze_optimization_2023}.
	
	The intersection $\PCP_G\cap\CCP_G$ is never empty, since both contain the origin of $\SS(G)$ as $\Ob\in\CPP$ whenever the ground cone is closed. In fact, we always have $\PCP_G\supseteq \CCP_{G}$, since a completable matrix necessarily fulfills the condition stated in $\PCP_{G}$.
	From this, we have that
	\begin{align}\label{eqn:SparseRelaxationGap}
		\begin{split}		
			\hspace{-1cm}
			\min_{\Xb_{G}}
			\lrbr{
				\Cb_G\bullet \Xb_{G} 
				\colon
				\begin{array}{l}
					\AA_G(\Xb_{G}) = \b, \\ 
					\Xb_G\in\CCP_{G}(\KK)
				\end{array} 
			}		
			\geq 
			\min_{\Xb_{G}}
			\lrbr{
				\Cb_G\bullet \Xb_{G} 
				\colon 
				\begin{array}{l}
					\AA_G(\Xb_{G}) = \b, \\ 
					\Xb_G\in\PCP_{G}(\KK)
				\end{array} 
			},
		\end{split}
	\end{align}
	and the majority of approaches to sparse reformulations are concerned with closing the relaxation gap in \eqref{eqn:SparseRelaxationGap} by identifying conditions on $G$ under which  $\PCP_G = \CCP_G$. The classical results are summarized as follows:
	\begin{prop}\label{prop:KnownResults}
		The following statements hold: 
		\begin{itemize}
			\item $\PCP_G(\R^n) = \CCP_G(\R^n)$ if and only if $G$ is a chordal graph.
			\item $\PCP_G(\R^n_+) = \CCP_G(\R^n_+)$ if and only if $G$ is a block-clique graph. 
		\end{itemize}
	\end{prop}
	\begin{proof}
		This proposition merely summarizes the main results in \cite{drew_completely_1998,grone_positive_1984}. 
	\end{proof}
	Graphs are chordal if every induced cycle is a triangle. A graph is a block-clique if every block is a clique and a block is a subgraph that is connected, has no subgraph that can become disconnected by removing just one vertex and its adjacent edges (i.e.\ a cut vertex), and is not contained in any other subgraph with these two properties. For a detailed explanation of chordal graphs and block clique graphs, the reader is referred to the references given in the proof above; a concise summary of these definitions and the theory of matrix completion can be found in \cite{berman_completely_2003}. We will not delve into the details here, as they are not important for the rest of the text. 
	
	\subsection{From sparse reformulations to matrix completion}
	
	It was noted in \cite{gabl_sparse_2023} that requiring $\PCP_G = \CCP_G$ is excessively strong since all that is needed to close the gap in \eqref{eqn:SparseRelaxationGap} is that the optimal set of the relaxation contains a single completable matrix. Hence, alternative conditions for the completablility of a matrix in $\PCP_G$ are desirable. The main objective of this text is to identify such conditions by working in the opposite direction. Focusing on sparse matrices with arrow-head sparsity patterns with thin diagonals, we will prove sufficient conditions under which the gap in \eqref{eqn:SparseRelaxationGap} vanishes. This will then be used to derive a new sufficient condition for the completability of such matrices with respect to certain set-completely positive matrices cones. The bridge between closing the relaxation gap in \eqref{eqn:SparseRelaxationGap} and obtaining a completability result is established in the following theorem.  
	\begin{lem}\label{lem:closures}
		Let $\CC_1$ and $\CC_2$ be two convex sets in $\R^n$. If $\inf_{\x\in\CC_1} \b\T\x  = \inf_{\x\in\CC_2} \b\T\x$ for all $\b\in\R^n$ then $\mathrm{cl}(\CC_1) =  \mathrm{cl}(\CC_2)$. 
	\end{lem}
	\begin{proof}
		Assume there was a point $\z\in\mathrm{cl}(\CC_1)\setminus\mathrm{cl}(\CC_2)$. Then by strong convex separation there are  $\b\in\R^n\setminus\lrbr{\oo},\ \mu\in\R$ such that  $\b\T\z < \mu \leq \b\T\x,\ \forall \x\in \CC_2$  which implies the contradiction $\inf_{\x\in\CC_2} \b\T\x > \inf_{\x\in\CC_1} \b\T\x$. The analogous argument applies if $\z\in\mathrm{cl}(\CC_2)\setminus\mathrm{cl}(\CC_1)$.
	\end{proof}
	
	\begin{thm}\label{thm:ZeroGapimpliesComp}
		Assume that $\AA_G, \ \b$ and $G$ are such that the inequality in  \eqref{eqn:SparseRelaxationGap} holds with equality for all $\Cb_G\in\SS(G)$. Then $\AA_G(\Xb_G) = \b, \ \Xb_G\in\PCP_{G}(\KK) \implies \Xb_G\in\CCP_G(\KK)$.
	\end{thm}
	\begin{proof}
		Under our assumptions, the closures of the feasible sets of the two optimization problems in \eqref{eqn:SparseRelaxationGap} have to be identical, by \cref{lem:closures}. Thus, the theorem will follow if we can show that the set  $\lrbr{\Xb_{G}\in\SS(G) \colon \AA_G(\Xb_{G}) = \b,\ \Xb\in\CCP_{G}(\KK)}$ is closed, since then we get 
		\begin{align*}
			\hspace{-0.8cm}
			\lrbr{\Xb_{G}\in\SS(G) \colon \AA_G(\Xb_{G}) = \b,\ \Xb\in\PCP_{G}(\KK)}
			&\subseteq 
			\mathrm{cl}\lrbr{\Xb_{G}\in\SS(G) \colon \AA_G(\Xb_{G}) = \b,\ \Xb\in\PCP_{G}(\KK)}\\
			&=
			\mathrm{cl}\lrbr{\Xb_{G}\in\SS(G) \colon \AA_G(\Xb_{G}) = \b,\ \Xb\in\CCP_{G}(\KK)}\\
			&\subseteq \CCP_G(\KK).
		\end{align*}
		It suffices to show that $\CCP_{G}(\KK)$ is closed. But this follows immediately from the fact that $\CPP(\KK)$ is closed (see for example \cite[Proposition 11]{bomze_optimization_2023} or \cite[Lemma 1]{sturm_cones_2003}). To see this consider a sequence $\Xb_G^k\in\CCP_G(\KK), \ k \in \N$ with $\lim_{k\rightarrow \infty}\Xb_G^k = \Xb_G$. Then there is a sequence of completions $\Xb^k\in\CPP(\KK)$ for which we have $(\Xb^k)_{ij}\rightarrow(\Xb_G)_{ij}$ whenever $(i,j)$ is specified. In particular, $\diag(\Xb^k) \rightarrow \diag(\Xb_G)$ since all the diagonal entries are specified. Hence, the diagonal entries of the series of completions can eventually be bounded, which, from the fact that $\Xb^k\in\CPP(\KK)\subseteq \SS^n_+$, implies that all the entries are bounded eventually. Thus, we see that, perhaps after switching to a subsequence, we have $\Xb^k\rightarrow \Xb$ with $(\Xb)_{ij}=(\Xb_G)_{ij}$ for all $(i,j)\in\EE$ and $\diag(\Xb) = \diag(\Xb_G)$. But since $\CPP(\KK)$ is closed, $\Xb\in\CPP(\KK)$ is a set-completely positive completion of $\Xb_G$ so that \redd{$\Xb_G$} is an element of $\CCP_{G}(\KK)$, which completes the proof
	\end{proof}
	
	This theorem plays a central role in the derivation of our new completion result, as we will identify conditions on the problem data under which the gap in \eqref{eqn:SparseRelaxationGap} can be closed (see \cref{thm:DecomposableSpecialCaseMis1}), which \cref{thm:ZeroGapimpliesComp} then translates into sufficient conditions for completability (see \cref{thm:ConditionalCPPCompletion}). 
	
	\begin{rem}
		In \cite{vandenberghe_chordal_2015} the closedness of $\CCP_G$ was proved for the special case $\KK = \R^n$, where the proof strategy was based on \cite[Theorem 9.1]{rockafellar_convex_2015} and the fact that $\CCP_{G}$ can be regarded as a type of projection of $\CPP$. A similar strategy could have been used for the above theorem, but a direct proof is more accessible. 
	\end{rem}
	
	\subsection*{Contribution and outline}
	The rest of this article is organized as follows: In \cref{sec:Review of arrow-head sparsity} we summarize the most important facts about partial matrices with arrow-head specification patterns. Following that, we will, in \cref{sec:An exact sparse relaxation}, investigate a special case of \eqref{eqn:SparseRelaxationGap} with arrowhead sparsity where we can close the relaxation gap under some additional assumptions (see \cref{thm:DecomposableSpecialCaseMis1}). Based on this result and \cref{thm:ZeroGapimpliesComp}, we will, in \cref{sec:A new completion theorem}, identify matrices with arrow-head sparsity pattern and thin diagonal that are completable with respect to certain instances of $\CPP$ that include the traditional completely positive matrix cone as a special case (\cref{thm:ConditionalCPPCompletion}). The sufficient conditions of Theorems \ref{thm:DecomposableSpecialCaseMis1} and \ref{thm:ConditionalCPPCompletion} are thoroughly discussed in \cref{sec:Sufficient conditions for XX} in order to make these technical \redd{assumptions} more tangible to the reader. \blue{There, we also give some illustrative examples of our main result.} Finally, in \cref{sec:Application} we show that we can use the new completion result in order to derive ex-post certificates of tightness of sparse conic relaxation\redd{s} of \redd{inequality-constrained} quadratic conic problem\redd{s}.  
	
	\section{Partial matrices with arrow-head specification pattern}\label{sec:Review of arrow-head sparsity}
	Partial matrices with arrowhead specification patterns can be written as 
	\begin{align*}
		\hspace{-1cm}
		\begin{bmatrix}
			\Xb & \Zb_1\T & \Zb_2\T &\dots & \Zb_{S-1}\T & \Zb_S\T \\
			\Zb_1 & \Yb_{1} & \mathbf{*} & \dots  & \mathbf{*}  & \mathbf{*}\\
			\Zb_2 & \mathbf{*} & \Yb_{2}& \dots  & \mathbf{*}  & \mathbf{*} \\
			\vdots&\vdots&\vdots&\ddots & \vdots & \vdots  \\
			\Zb_{S-1}  & \mathbf{*}  & \mathbf{*} & \dots & \Yb_{S-1}  & \mathbf{*} \\
			\Zb_S  & \mathbf{*}  & \mathbf{*}&\dots  & \mathbf{*} & \Yb_{S}
		\end{bmatrix},\
	\end{align*}
	where $ \Xb\in\SS^{n_1},\ \Yb_{i}\in\SS^{n_2},\ \Zb_i\in\R^{n_2\times n_1}, \ i \in\irg{1}{S}.$ The fully specified submatrices are given by 
	\begin{align*}
		\begin{bmatrix}
			\Xb & \Zb_i\T\\ \Zb_i & \Yb_{i}
		\end{bmatrix}\in\SS^{n_1+n_2}, \ i \in \irg{1}{S}.
	\end{align*} 
	Such sparsity patterns may arise in conic reformulations of two-stage stochastic quadratic optimization problems as studied in \cite{bomze_two-stage_2022,gabl_sparse_2023}. The specification graph of such a partial matrix, say $G$, consists of  $S+1$ cliques, where the first one encompasses $n_1$ nodes, each of which is connected to all other nodes of $G$. The other $S$ cliques are of size $n_2$ and are not connected to each other. Henceforth, we will indicate the specification graph partial matrices with arrowhead sparsity as $G_{n_1,n_2}^S$. 
	
	If we relabel the nodes in such a graph its structure does not change, but the specification pattern of the associated partial matrices does. The effect on those matrices would be akin to permuting rows and columns in a symmetric fashion so that this pattern would no longer resemble an arrowhead. On the other hand, if the specification graph $G$ of a partial matrix $\Ab_G$ has a structure like $G_{n_1,n_2}^S$ then we can find a relabeling such that the specification pattern of the associated partial matrix $\Bb_{G_{n_1,n_2}^S}$ has arrowhead structure. If we find completion of the latter matrix with respect to $\CPP(\KK)$, we also get a completion of $\Ab_G$ with respect to $\CPP(\bar{\KK})$ where $\KK$ is obtained from $\bar{\KK}$ via the appropriate permutation of coordinates. Hence, our discussion applies more broadly to partial matrices whose specification graph has the same structure as $G_{n_1,n_2}^S$. However, the simplicity of our presentation would suffer greatly if we abandoned the arrowhead pattern as our point of reference. Hence, we focus on matrices with such patterns and trust that the reader can envision the aforementioned generalization themselves.  
	
	We have the following characterization of $\PP\CC\PP_{G_{n_1,n_2}^S}$ in case the ground cone decomposes into smaller cones in a convenient manner, and it will be instrumental in the succeeding sections.   
	
	\begin{prop}\label{prop:CharacterizePCP}
		Define $\KK \coloneqq \KK_0\times\KK_1\times\dots\times\KK_S$ for some \red{ground} cones $\KK_0\subseteq \R^{n_1}$ and $\KK_i\subseteq \R^{n_2},\ i\in\irg{1}{S}$. We have 
		\begin{align*}
			\PP\CC\PP_{G_{n_1,n_2}^S}(\KK) 
			= 
			\lrbr{
				\Mb_{G_{n_1,n_2}^S}\in\SS(G_{n_1,n_2}^S)
				\colon 
				\begin{bmatrix}
					\Xb & \Zb_i\T\\ \Zb_i & \Yb_{i}
				\end{bmatrix}\in\CPP(\KK_0\times\KK_i), \ i \in \irg{1}{S}
			}.
		\end{align*}
	\end{prop}
	\begin{proof}
		Let $G \coloneqq G_{n_1,n_2}^S$ for notational convenience.
		For any partial matrix $\Mb_{G}$ in the right hand side set, take the $i$-th fully specified principle submatrix $\Mb_i= \Xb_i\Xb_i\T$, where the factor  $\Xb_i\in(\KK_0\times\KK_i)^r$ for some $r\in\N$ certifies $\Mb_i\in\CPP(\KK_0\times\KK_i)$. Let $\II_i$ be the row and column indices of that submatrix relative to $\Mb_{G}$. We can augment $\Xb_i$ by adding rows of zeros in order to get a matrix in $\hat{\Xb}_i\in\KK^r$. More precisely, we can add \red{$(i-1)n_2$} rows of zeros after the first $n_1$ rows of $\Xb_i$ and also add $(S-i)n_2$ such rows at the end of $\Xb_i$. For the so constructed matrix we have $\hat{\Xb}_i\hat{\Xb}_i\T\in\CPP(\KK)$ since $\oo\in\KK_i, \ i\in\irg{1}{S}$ as they are closed cones. The submatrix induced by $\II_i$ is $\Mb_i$ as desired. The converse follows from the fact that for any matrix $\Mb\in\CPP(\KK)$ with certificate $\Mb = \Xb\Xb\T,\ \Xb\in\KK^r$ we may delete the aforementioned rows from $\Xb$ to obtain an $\Xb_i\in\KK_0\times\KK_i$ which certifies that the submatrix $\Mb_i = \Xb_i\Xb_i\T$ of $\Mb_G$ induced by $\II_i$ is a member of $\CPP(\KK_0\times\KK_i)$.
	\end{proof}
	Note that $\KK \in  \lrbr{\R_+^{n_1+n_2S},\R^{n_1+n_2S}}$ can always be cast as a Cartesian product so that in the positive semidefinite/completely positive case the characterization requires that the fully specified submatrices are positive semidefinite/completely positive as per usual. The following lemma was proved in \cite{gabl_sparse_2023}:
	\begin{lem}\label{lem:ArrowHeadMatricesAreChordal}
		The specification graph $G_{n_1,n_2}^S$ is chordal. If $S>1$, then $G_{n_1,n_2}^S$ is a block-clique graph if and only if $n_1\in\lrbr{0,1}$. 
	\end{lem}
	The important take away from this result is that in light of \cref{prop:KnownResults} we have $\PCP_{G_{n_1,n_2}^S}(\R^n) = \CCP_{G_{n_1,n_2}^S}(\R^n)$ but $\PCP_{G_{n_1,n_2}^S}(\R^n_+) \supset \CCP_{G_{n_1,n_2}^S}(\R^n_+)$ unless $n_1\in\lrbr{0,1}$; we used the abbreviation $n= n_1+Sn_2$. The main contribution of our paper will be a result on sufficient conditions for matrices with arrow-head specification pattern to be completable with respect to a cone that is more general than just $\CPP(\R^{n}_+)$ but where $n_2 = 1$.
	
	As discussed above, not all partial matrices of this kind are completable. In fact the classical example of a partial matrix in $\PCP_{G_{n_1,n_2}^S}(\R^n_+)\setminus \CCP_{G_{n_1,n_2}^S}(\R^n_+)$ first presented in \cite{berman_completely_2003} is given by an arrow-head partial matrix with \red{width equal to one}. For the reader's convenience, we repeat this example here: 
	
	\begin{exmp}\label{exmp:noncompletableExample} (\cite[Example 2.23]{berman_completely_2003})
		Consider the following partial matrix and its fully specified principal submatrices
		\begin{align*}
			\Ab_{G_{2,1}^2} \coloneqq 
			\begin{bmatrix}
				6&3&0&3\\
				3&6&3&0\\
				0&3&2&*\\
				3&0&*&2\\
			\end{bmatrix},\
			\Ab_1 \coloneqq 
			\begin{bmatrix}
				6&3&0\\
				3&6&3\\
				0&3&2\\
			\end{bmatrix}, \quad
			\Ab_2 \coloneqq 
			\begin{bmatrix}
				6&3&3\\
				3&6&0\\
				3&0&2\\
			\end{bmatrix},
		\end{align*}
		which has an arrow-head sparsity pattern with specification graph $G_{2,1}^2$ (note that compared to the original presentation in \cite{berman_completely_2003} we have permuted some rows and columns to get the arrowhead structure). The determinant of any completion, say $\Ab$, is given by $\det\left(\Ab\right) = -27\left((\Ab)_{4,3}+1\right)^2 < 0$ since $(\Ab)_{4,3}\geq 0$ so that it cannot be completed to an element of $\CPP(\R^{4}_+)\subseteq \SS^{4}_+$. However, the fully specified submatrices $\Ab_1,\ \Ab_2$ are easily shown to be positive semidefinite and therefore completely positive since $\CPP(\R^n) = \SS^n_+\cap\NN^n$ for $n\leq 4$ (see \cite{maxfield_matrix_1962}). Thus, $\Ab_{G_{2,1}^2}\in\PCP_{G_{2,1}^2}(\R^4_+)\setminus \CCP_{G_{2,1}^2}(\R^4_+)$.
	\end{exmp}

	\section{A conditionally exact sparse relaxation}
	\label{sec:An exact sparse relaxation}

	Our goal in this section is to provide a $\PCP$-based, conditionally tight, relaxation of the following quadratic optimization problem:
	\begin{align}\label{eqn:DecomposableSpecialCaseMis1}
		\begin{split}
			\inf_{\x,\y_1,\dots,\y_S} 
			\lrbr{
				\x\T\Ab\x + \aa\T\x+ 
				\sum_{i=1}^{S}
				\redd{
					\left(
					\x\T\Bb_i\y_i + \y_i\T\Cb_i\y_i+\cc_i\T\y_i
					\right)}
				\colon 
				\begin{array}{l}
					\red{\f_0\T\x = d_0,\ \x\in \KK_0,}\\
					\f_i\T\x+\g_i\T\y_i = \red{d_i},\ \ i \in\irg{1}{S},\\					
					\y_i\in\KK_i,\ i\in\irg{1}{S},		
				\end{array}	
			}, 
		\end{split}
	\end{align}
	where $\KK_0\subseteq\R^{n_x},\ \KK_i\subseteq \R^{n_y}, \ i \in\irg{1}{S}$ are \red{ground} cones. Note, that we allow for $\oo\in\lrbr{\f_0,\dots,\f_S}$ so that the respective constraint can be rendered vacuous on $\x$. Of course, if we set $\f_0=\oo$, then we should also set $d_0=0$ lest the problem becomes trivially infeasible. We will henceforth denote the feasible set of this problem as $\FF$. Define $\KK =\R_+\times\KK_0\times\KK_1\times \dots,\times\KK_S$, then we see that due to \eqref{eqn:chlifting} and \eqref{eqn:characterizationofGGFF}, the above problem has a reformulation \red{based on $\GG(\FF)$} as a conic problem over $\CPP(\KK)\subseteq \SS^{n_x+Sn_y+1}$, given by 
	\begin{align}\label{eqn:FullReformulation}
		\hspace{-1cm}
		\inf_{
			\substack{\Xb,\x,\Yb_{i,i},\\\Zb_i,\y_i}			
		}
		\lrbr{
			\begin{array}{l}
				\Ab\bullet\Xb+\aa\T\x + 
				\sum_{i=1}^{S}  
				\redd{
					\left(
					\Bb_i\bullet\Zb_i + \Cb_i\bullet \Yb_{i,i}+\cc_i\T \y_i
					\right)}
				\colon \\
				\begin{array}{rl}					
					\red{
						\f_0\T\x = d_0,\
						\f_0\f_0\T\bullet\Xb = d_0^2,} & \quad \\
					\f_i\T\x + \g_i\red{\T}\y_i = \red{d_i},& i \in\irg{1}{S},\\
					\f_i\f_i\T\bullet\Xb + 2\f_i\g_i\T\Zb_i+\g_i\g_i\T \Yb_{i,i} = \red{d_i^2},& i  \in\irg{1}{S},		
				\end{array}						
			\end{array} 					
			\begin{bmatrix}
				1 & \x\T &\y_1\T & \dots & \y_S\T \\
				\x& \Xb & \Zb_1\T & \dots & \Zb_S\T \\
				\y_1&\Zb_1 & \Yb_{1,1} & \dots & \Yb_{1,S}\\
				\vdots&\vdots&\vdots&\ddots & \vdots \\
				\y_S & \Zb_S & \Yb_{S,1}&\dots & \Yb_{S,S}
			\end{bmatrix}  \in \CPP\left(\KK\right)
		}.				
	\end{align}
	Note that the matrix variables $\Yb_{i,j}$ for $i\neq j$ do not appear outside the conic constraint. Thus, we can declare them to be unspecified so that the entire matrix block becomes an element of $\SS(G^S_{n_x+1,n_y})$ and replace $\CPP$ by $\CCP_{G^S_{n_x+1,n_y}}$ to obtain an equivalent reformulation in the space of partial matrices courtesy of \eqref{eqn:SparseReformulation}.  We will show that, \redd{under certain conditions,} the respective relaxation based on $\PCP_{G^S_{n_x+1,n_y}}$ is exact by directly showing that it is equivalent to  \eqref{eqn:DecomposableSpecialCaseMis1}. 
	
	Our theorem involves \red{assumptions} on certain sets associated with the above optimization problem \red{(see assumptions i.-iii.\ in \cref{thm:DecomposableSpecialCaseMis1} below)}. \red{These} \red{assumptions} may seem elusive at first, but we will shed some light on \redd{them} in \cref{sec:Sufficient conditions for XX}. In this way, we also avoid overloading the already lengthy proof of this section's main result with technical details regarding said assumptions. Before we can prove the main result of this section we need to establish the following lemma:
	
	\begin{lem}\label{thm:GeneralizedStQPReformulation}
		Let $\KK$ be a \red{ground} cone, $\aa\in\mathrm{int}(\KK^*)$ and $d\geq0$. Then the set $\FF_{l} \coloneqq \lrbr{\Xb\in \CPP(\KK) \colon \aa\aa\T\bullet \Xb = \red{d}}$ is \redd{ nonempty,} compact and its extreme points are matrices of the form $\x\x\T$ \red{with $\x\in\KK$}.
	\end{lem}
	\begin{proof}
		First, note that $\aa\in\mathrm{int}(\KK^*)$ implies that $\aa\aa\T\bullet\v\v\T = (\v\T\aa)^2>0,\  \forall \v\in\KK\setminus\lrbr{\oo},$ so that by the definition of $\CPP(\KK)$ we also have  $\aa\aa\T\bullet\Vb>0\ \forall\, \Vb\in\CPP(\KK)\setminus\lrbr{\Ob}$. If $d=0$, then the latter condition implies that $\FF_{l} = \lrbr{\Ob}$, but $\Ob = \oo\oo\T$ and $\oo\in\KK$ as it is closed. This also shows that $\FF_l$ is compact \redd{by \cref{prop:ReccCones} a)}, since, firstly, for any $d\geq 0$ we have $\FF_l\neq\emptyset$ \blue{(take any $\v\in\KK\setminus\lrbr{\oo}$ so that $\aa\T\v>0$ and $d(\aa\T\v)^{-2}\v\v\T\in\CPP(\KK)$, hence $d(\aa\T\v)^{-2}\aa\aa\bullet\v\v\T = d$)}, so that, secondly,  \cref{prop:ReccCones} e) further implies that the recession cone $\FF_l^{\infty}$ is identical to $\FF_l$ with $d=0$. So assume $d>0$ so that w.l.o.g we have $\red{d} = 1$, perhaps after rescaling $\aa$ by $\redd{1/\sqrt{d}}>0$, which would not impede $\aa$'s containment in $\mathrm{int}(\KK^*)$. 
		If $\Yb$ is an extreme matrix of $\FF_{l}$, then  $\Yb\in\CPP(\KK)$ and by \cite[Corollary 18.5.2]{rockafellar_convex_2015} we have $\Yb  = \sum_{i=1}^{k}\Yb_i,\ i\in \irg{1}{k}$ for some nonzero generators $\Yb_i,\ i \in \irg{1}{k}$ of extreme rays of $\CPP(\KK)$. Since $\aa\aa\T \bullet \Yb_i>0,$ there are $\mu_i\coloneqq  1/\left(\aa\aa\T\bullet\Yb_i\right)$ with $\aa\aa\T\bullet\left(\mu_i\Yb_i\right) = 1$. This implies 
		$\Yb = \sum_{i=1}^{k}\Yb_i = \sum_{i=1}^{k}\tfrac{1}{\mu_i}\mu_i\Yb_i$ 
		where 
		$\mu_i\Yb_i \in\FF_{l}$ 
		and 
		$
		\sum_{i=1}^{k}\tfrac{1}{\mu_i} 
		= 
		\sum_{i=1}^{k}\left(\aa\aa\T\bullet \Yb_i\right) 
		= 
		\aa\aa\T\bullet \sum_{i=1}^{k}\Yb_i =
		\aa\aa\T\bullet\Yb = 1
		$ 
		so that $\Yb = \mu_i\Yb_i, \ i \in\irg{1}{S}$ by extremality, as extreme points of a convex set cannot be cast as convex combinations of members of that set different from themselves. But since each $\Yb_i$ is a generator of an extreme ray in $\CPP(\KK)$, they are of the form $\x\x\T$ \red{with $\x\in\KK$}, hence, so is $\Yb$.  	
	\end{proof}
	
	\begin{thm}\label{thm:DecomposableSpecialCaseMis1}
		Consider the problem \eqref{eqn:DecomposableSpecialCaseMis1} where $\KK_0\subseteq\R^{n_x},\ \KK_i\subseteq \R^{n_y}, \ i \in\irg{1}{S}$ are \red{ground} cones. The following optimization problem gives a lower bound \red{on \eqref{eqn:DecomposableSpecialCaseMis1}:}
		\begin{align}\label{eqn:DecomposableSpecialCaseMis1Ref}	
			\hspace{-1cm}		
			\inf_{
				\substack{
					\Xb,\x,\Yb_i,\\ \Zb_i.\y_i
				}
			} 
			\lrbr{
				\begin{array}{l}
					\Ab\bullet\Xb+\aa\T\x + 
					\sum_{i=1}^{S} 
					\redd{
						\left(
						\Bb_i\bullet\Zb_i + \Cb_i\bullet\Yb_i+\cc_i\T\y_i
						\right)}
					\colon 
					\\
					\begin{array}{rl}
						\red{\f_0\T\x = d_0,\
							\f_0\f_0\T\bullet\Xb = d_0^2}, & \quad \\
						\f_i\T\x + \g_i\T\y_i = \red{d_i}, &i \in\irg{1}{S},\\
						\f_i\f_i\T\bullet\Xb + 2\ \f_i\g_i\T\bullet\Zb_i+\g_i\g_i\T \bullet\Yb_i = \red{d_i^2}, & i \in\irg{1}{S},
					\end{array}			
				\end{array}
				\begin{array}{r}
					\begin{bmatrix}
						1 & \x\T &\y_i\T \\
						\x& \Xb  &\Zb_i\T\\
						\y_i& \Zb_i& \Yb_i
					\end{bmatrix} \in \CPP(\R_+\times\KK_0\times\KK_i),\\ i  \in\irg{1}{S}
				\end{array}
			}.	
		\end{align}		
		In addition, if we assume that $\cc_i = \beta_i\g_i,\ \Bb_i = \b_i\g_i\T,\ \b_i\in\R^{n_x},\ \beta_i\in\R,\ i \in \irg{1}{S} $ and that for the sets 
		\begin{align*}				
			\begin{array}{ll}									
				\hspace{-1cm}
				\FF_i  \coloneqq \lrbr{[\x\T,\y_i\T]\T\in\KK_0\times\KK_i\colon \f_i\T\x+\g_i\T\y_i = \red{d_i}},\ i \in \irg{1}{S},
				&
				\red{\FF_0}  \coloneqq \red{ \lrbr{\x\in\KK_0\colon \f_0\T\x = d_0},}
				\\
				\hspace{-1cm}
				\XX_i \red{\coloneqq} \lrbr{
					(\x,\Xb) \colon \exists (\y_i,\Yb_i,\Zb_i) \colon
					\begin{bmatrix}
						1 & \x\T &\y_i\T \\
						\x& \Xb  &\Zb_i\T\\
						\y_i& \Zb_i& \Yb_i
					\end{bmatrix} \in \GG\left(\FF_i\right)
				}, \ i \in \irg{1}{S}, 
				&
				\XX_0 \red{\coloneqq} \lrbr{
					(\x,\Xb) \colon 
					\begin{bmatrix}
						1 & \x\T \\
						\x& \Xb  				
					\end{bmatrix} \in \GG\left(\FF_0\right)
				},		
			\end{array}
		\end{align*}
		it holds that  
		\begin{enumerate}
			\item[i.] the vectors \red{$\g_i\in\mathrm{int}(\KK_i^*),\ i\in\irg{1}{S}$}, 
			\item[ii.] the set $\XX\coloneqq\bigcap_{i = 0}^S\XX_i$ is bounded, and 
			\item[iii.]  \redd{the extreme points of $\XX$} are of the form $(\x,\x\x\T),$
		\end{enumerate}
		then the bound is exact.
	\end{thm}
	\begin{proof}
		It is clear that the conic optimization problem gives a lower bound since for any \redd{feasible} solution of the original QP we can set $\Xb =\x\x\T,\ \Yb_i = \y_i\y_i\T, \ \Zb_i =\y_i\x\T, i \in\irg{1}{S}$, \blue{which is feasible for the relaxation and yields the same objective function value}. Alternatively, one can see that \eqref{eqn:DecomposableSpecialCaseMis1Ref} is the relaxation of \eqref{eqn:FullReformulation} based on $\PCP_{G^S_{n_x+1,n_y}}$ and the characterization of that cone presented in \cref{prop:CharacterizePCP}.
		
		\blue{
			For the converse we will first argue that, under the additional assumptions, we may focus on the case where both \eqref{eqn:DecomposableSpecialCaseMis1} and \eqref{eqn:DecomposableSpecialCaseMis1Ref} are feasible and where $\Bb_i = \Ob$ and $\cc_i = \oo, \ i\in\irg{1}{S}$. So let us start by showing that \eqref{eqn:DecomposableSpecialCaseMis1} is feasible if and only if \eqref{eqn:DecomposableSpecialCaseMis1Ref} is feasible. Since we argued in the preceding paragraph that \eqref{eqn:DecomposableSpecialCaseMis1Ref} is a relaxation of \eqref{eqn:DecomposableSpecialCaseMis1} the "only if"-part is clear. For the "if"-direction assume that \eqref{eqn:DecomposableSpecialCaseMis1Ref} is feasible and consider the fact that 
			from \eqref{eqn:characterizationofGGFF} we have the characterizations 
			\begin{align}\label{eqn:CharacterizationGGofFFi}
				\GG\left(\FF_i\right)  = 
				\lrbr{				
					\begin{bmatrix}
						1 & \x\T &\y_i\T \\
						\x& \Xb  &\Zb_i\T\\
						\y_i& \Zb_i& \Yb_i
					\end{bmatrix} \in \CPP\left(\R_+\times \KK_0\times\KK_i\right)
					\colon
					\begin{array}{r}
						\f_i\T\x + \g_i\T\y_i = d_i, \\
						\f_i\f_i\T\bullet\Xb + 2\ \f_i\g_i\T\bullet\Zb_i+\g\g_i\T \bullet\Yb_i = d_i^2,
					\end{array} 					
				},
			\end{align}
			for all $i\in\irg{1}{S}$. Then, by \cref{prop:LinearPartinKK} (applied to each of the matrix blocks individually), we see that the values of $\x,\y_1,\dots,\y_S$ from the northeast blocks of the feasible solution for \eqref{eqn:DecomposableSpecialCaseMis1Ref} are feasible for \eqref{eqn:DecomposableSpecialCaseMis1} so that it is feasible as well. Thus, we may focus on the case where both problems are feasible since if either of them is infeasible they both evaluate to $\infty$, so that they are indeed equivalent. Note that this entails that $\FF_i, \ \GG(\FF_i)$ and therefore also $\XX_i$ are nonempty for any $i\in\irg{1}{S}$, but also $\XX$ is nonempty since any feasible solution for \eqref{eqn:DecomposableSpecialCaseMis1Ref} can be projected onto its $(\x,\Xb)$-coordinates to obtain an element of $\XX$. }
		
		Further, under the additional assumptions on $\cc_i,\ \Bb_i, \ i \in\irg{1}{S}$, we can always transform the objective function 
		\begin{align*}
			&\ \x\T\Ab\x + \aa\T\x + 
			\sum_{i=1}^{S} 
			\redd{
				\left(
				\x\T\b_i\g_i\T\y_i + \y_i\T\Cb\y_i+\beta_i\g_i\T\y_i
				\right)}
			\\
			=&\ \x\T\Ab\x + \aa\T\x + 
			\sum_{i=1}^{S} 
			\left(
			\x\T\b_i\left(d_i-\f_i\T\x\right) + \y_i\T\Cb\y_i+ \beta_i(d_i-\f_i\T\x)
			\right)
			\\
			= &\ 
			\x\T\left(\Ab-0.5\left(\sum_{i=1}^{S}\b_i\f_i\T+\f_i\b_i\T\right)\right)\x 
			+ 
			\redd{
				\left(\aa+\sum_{i=1}^{S}\left(d_i\b_i-\beta_i\f_i\right)\right)\T\x }
			+ 
			\sum_{i=1}^{S} 
			\redd{
				\left(
				\y_i\T\Cb\y_i+d_i\beta_i,
				\right)}
		\end{align*}
		so we only need to consider the case where $\Bb_i = \Ob,\ \cc_i = \oo,\ i\in\irg{1}{S}$.
		\blue{ Caution must be taken here: if the objective function in \eqref{eqn:DecomposableSpecialCaseMis1} is replaced by the above reformulation, then the respective relaxation steps lead to a version of \eqref{eqn:DecomposableSpecialCaseMis1Ref} that has a different objective function than what is depicted in the statement of the theorem, i.e., instead of $\Ab$ we would have $\Ab-0.5\left(\sum_{i=1}^{S}\b_i\f_i\T+\f_i\b_i\T\right)$, and so on and so forth. 
		On the other hand, if these reformulation steps were applied to the original formulation of \eqref{eqn:DecomposableSpecialCaseMis1} we would get \eqref{eqn:DecomposableSpecialCaseMis1Ref} as depicted in the statement of the theorem. It is not obvious that these two versions of \eqref{eqn:DecomposableSpecialCaseMis1Ref} are equivalent, and we will take care of this at the end of the proof.}
		
		\blue{With these two specifications in mind, we start out our main argument by defining}
		\begin{align}\label{eqn:PhixX}
			\phi_i(\x,\Xb) \coloneqq \red{\inf_{\y_i,\Yb_i,\Zb_i}}\lrbr{ \Cb\bullet\Yb_i \colon	\begin{bmatrix}
					1 & \x\T &\y\T_i \\
					\x& \Xb  &\Zb\T_i\\
					\y_i& \Zb_i& \Yb_i
				\end{bmatrix} \in \GG\left(\FF_i\right)
			},\quad i \in \irg{1}{S},
		\end{align}	
		which are convex functions that are finite on $\XX_i,\ i \in\irg{1}{S}$ respectively and $\infty$ everywhere else. The convexity is easily checked by plugging in a convex combination of two points and realizing that the convex combinations of the two minimizers associated with each of the points yield an upper bound for the minimization problem since $\GG(\FF_i),\ i \in \irg{i}{S}$ are convex. If the problem is infeasible, i.e., \ $(\x,\Xb)\notin\XX_i$, then the functions evaluate to $\infty$, which does not impede convexity. We can therefore \redd{reformulate} (\ref{eqn:DecomposableSpecialCaseMis1Ref}) as 
		\begin{align}\label{eqn:PhiReformulation}
			\min_{(\x,\Xb)\in\XX}\Ab\bullet\Xb+ \aa\T\x + \sum_{i=1}^{S} \phi_i(\x,\Xb),
		\end{align} 
		using the definition of $\red{\XX \coloneqq \bigcap_{i=0}^S\XX_i}$ and the characterization of $\GG(\FF_i)$ in \eqref{eqn:CharacterizationGGofFFi}.
		
		\blue{
			From here on, the proof strategy is to use the optimal solution of a certain relaxation of \eqref{eqn:PhiReformulation} in order to construct a feasible solution of \eqref{eqn:DecomposableSpecialCaseMis1} with the same objective value as said relaxation. To this end we will, firstly, construct underestimators of $\phi_i,\ i \in\irg{1}{S}$ which will on one hand turn out to be affine in $(\x,\Xb)$ and on the other hand also supply vectors $\bar{\y}_i,\ i \in\irg{1}{S}$, which will be crucial in building the aforementioned feasible solution. These underestimators are constructed via a certain relaxation of the feasible sets of the infimum-problems defining $\phi_i,\ i \in\irg{1}{S}$. Secondly, replacing $\phi_i, \ i \in\irg{1}{S}$ by these underestimators in \eqref{eqn:PhiReformulation} will yield a linear optimization problem over the convex set $\XX$, which we will eventually argue to be not only bounded but compact, so that the optimal solution is attained at an extreme point $(\bar{\x},\bar{\x}\bar{\x}\T)$. Finally, the so constructed vector $\bar{\x}$ together with appropriately scaled versions of the vectors $\bar\y_i, \ i \in\irg{1}{S}$ will then be used to construct the desired upper bound. 
		}
		
		Now, fix any $i\in\irg{1}{S}$. In the following paragraphs, we suppress the $i$-indices of the variables momentarily to avoid double indices. \blue{The values for $\Yb$ that are feasible for the optimization problem defining $\phi_i(\x,\Xb)$ in (\ref{eqn:PhixX}) are exactly the members of the set }
		\begin{align*}
			\YY_i(\x,\Xb) \coloneqq \lrbr{ \Yb \colon \exists(\Zb,\y)\colon 
				\begin{bmatrix}
					1 & \x\T &\y\T \\
					\x& \Xb  &\Zb\T\\
					\y& \Zb& \Yb
				\end{bmatrix} \in \GG\left(\FF_i\right)
			}.
		\end{align*}
		\blue{To construct the underestimator of $\phi_i$  mentioned above, we define the set}
		\begin{align*}
			\bar{\YY}_i(\x,\Xb)
			\coloneqq 
			\lrbr{
				\Yb\in\CPP\left(\KK_i\right) \colon \g_i\g_i\T\bullet\Yb = \red{d_i^2} - 2\red{d_i}\f_i\T\x+\f_i\f_i\T\bullet\Xb
			},
		\end{align*}
		\blue{which we now show to yield a relaxation of the feasible set of the infimum problems defining $\phi_i$, in other words: we show $\YY_i(\x,\Xb) \subseteq \bar{\YY}_i(\x,\Xb)$. }
		
		Note, that $(\x,\Xb)\notin\XX_i$ if and only if $\YY_i(\x,\Xb) = \emptyset$ in which case the containment is trivial. So assume otherwise and pick $\Yb\in\YY_i(\x,\Xb)$. Then, there exists a $(\Zb,\y)$ that completes $(\Yb,\Xb,\x)$ to an element of $\GG\left(\FF_i\right)$.
		Again, consider the characterization of $\GG(\FF_i)$ depicted in \eqref{eqn:CharacterizationGGofFFi}. 
		This set contains two types of elements: convex combinations of matrices of the form $\z\z\T\colon \z\in\lrbr{1}\times\FF_i$ and the cluster points of sequences of such convex combinations. We firstly discuss the case where $\Yb$ is such that the matrix block in $\GG(\FF_i)$ is of the first type\redd{, i.e.,} where we have a representation
		\vspace{-0.1cm}
		\begin{align}\label{eqn:ConvexCombinationRepresentationofGGFFi}
			\begin{bmatrix}
				1 & \x\T &\y\T \\
				\x& \Xb  &\Zb\T\\
				\y& \Zb& \Yb
			\end{bmatrix} = \sum_{\red{l}=1}^{k} \gl_l \begin{bmatrix}
				1\\\x_l\\\y_l
			\end{bmatrix}
			\begin{bmatrix}
				1\\\x_l\\\y_l
			\end{bmatrix}\T \colon\ \f_i\T\x_l+\g_i\T\y_l = \red{d_i},\ \x_l\in\KK_0,\ \red{\y_l}\in\KK_i, \ \red{l\in\irg{1}{k}},
		\end{align}	
		for some $\ \gl_l\geq 0,\  l\in\irg{1}{k},\ \sum_{\red{l=1}}^{k}\gl_l = 1$.
		This allows us to perform the following reformulation: 
		\begin{align*}
			\f_i\f_i\T\bullet\Xb + 2\ \f_i\g_i\T\bullet\Zb+\g_i\g_i\T \bullet\Yb &= \red{d_i^2},\\
			\f_i\f_i\T\bullet\sum_{\red{l=1}}^{k}\gl_l\x_l\x_l\T + 2\f_i\g_i\T\bullet\sum_{\red{l=1}}^{k}\gl_l\y_l\x_l\T+\g_i\g_i\T\bullet\sum_{\red{l=1}}^{k}\gl_l\y_l\y_l\T &= \red{d_i^2}, \\
			\sum_{\red{l=1}}^{k}\gl_l\left(\f_i\T\x_l\right)^2 + 2\sum_{\red{l=1}}^{k}\gl_l \left(\f_i\T\x_l\right)\left(\g_i\T\y_l\right) +\sum_{\red{l=1}}^{k}\gl_l\left(\g_i\T\y_l\right)^2 &= \red{d_i^2}, \\
			\sum_{\red{l=1}}^{k}\gl_l\left(\f_i\T\x_l\right)^2 + 2\sum_{\red{l=1}}^{k}\gl_l \left(\f_i\T\x_l\right)\left(\red{d_i}-\f_i\T\x_l\right) +\sum_{\red{l=1}}^{k}\gl_l\left(\g_i\T\y_l\right)^2 &= \red{d_i^2}, \\
			\sum_{\red{l=1}}^{k}\gl_l\left(\f_i\T\x_l\right)^2 + 2
			\left(\sum_{\red{l=1}}^{k}\gl_l \left(\red{d_i}\f_i\T\x_l\right) -\sum_{\red{l=1}}^{k}\gl_l \left(\f_i\T\x_l\right)^2 \right) +\sum_{\red{l=1}}^{k}\gl_l\left(\g_i\T\y_l\right)^2 &= \red{d_i^2}, \\
			2\red{d_i}\f_i\T\x-\f_i\f_i\T\bullet\Xb +\g_i\g_i\T\bullet \Yb &=\red{d_i^2}.
		\end{align*}	
		\redd{We} also have $\Yb\in\CPP\left(\KK_i\right)$ from the conic constraint by \eqref{eqn:SubconesOfCPP}, so that \red{we} see that $\Yb\in\bar{\YY}_i(\x,\Xb)$. To cover the second case, we assume that $\Yb\in\YY_i(\x,\Xb)$ is such that the respective matrix block in $\GG(\FF_i)$ is a limit point of a sequence of convex combinations of the form described in \eqref{eqn:ConvexCombinationRepresentationofGGFFi}\redd{, i.e.,} 
		\begin{align*}
			\begin{bmatrix}
				1 & \x\T &\y\T \\
				\x& \Xb  &\Zb\T\\
				\y& \Zb& \Yb
			\end{bmatrix} = 
			\lim\limits_{r\rightarrow \infty}
			\begin{bmatrix}
				1 & \x_r\T &\y_r\T \\
				\x_r& \Xb_r  &\Zb_r\T\\
				\y_r& \Zb_r& \Yb_r
			\end{bmatrix}, \mbox{ such that every member can be decomposed as in  \eqref{eqn:ConvexCombinationRepresentationofGGFFi}}.
		\end{align*}
		Repeating the steps from the first case for each member of the sequence yields $2d_i\f_i\T\x_r-\f_i\f_i\T\bullet\Xb_r +\g_i\g_i\T\bullet \Yb_r =d_i^2,$ so that by continuity of linear functions we have 	
		$2d_i\f_i\T\x-\f_i\f_i\T\bullet\Xb +\g_i\g_i\T\bullet \Yb = \lim\limits_{r\rightarrow\infty}\left(2d_i\f_i\T\x_r-\f_i\f_i\T\bullet\Xb_r +\g_i\g_i\T\bullet \Yb_r\right) = \lim\limits_{r\rightarrow\infty}d_i^2 = d_i^2$ and $\Yb\in\CPP(\KK_i)$ since \redd{$\Yb_r, \ r\in\N$ are members} of that cone by \eqref{eqn:SubconesOfCPP} and since \redd{that} cone is also closed as the ground cone $\KK_i$ is closed. In both cases we also get $\red{d_i^2}-2\red{d_i}\f_i\T\x+\f_i\f_i\T\bullet\Xb = \g_i\g_i\T\bullet\Yb\geq 0$, since $\Yb\in\CPP(\KK_i)$ and $\g\in\mathrm{int}(\KK_i^*)$ \redd{by assumption i.}.
		
		\blue{We resume using the i-indices on variables as no more double indices are imminent.} Whenever $(\x,\Xb)\in\XX$ \redd{we have that} $\bar{\YY}_i(\x,\Xb)$ is nonempty, \redd{compact} and its extreme points are of the form $\y_i\y_i\T$ with $\y_i\in\KK_i$, by \cref{thm:GeneralizedStQPReformulation}, which is applicable since $\g_i\in\mathrm{int}(\KK_i^*)$ \redd{by assumption} and $\red{d_i^2}-2\red{d_i}\f_i\T\x+\f_i\f_i\T\bullet\Xb\geq 0$. Since linear functions attain their infimum over compact convex sets at extreme points of that set \blue{there is a dyadic matrix } $\redd{\bar{\y}_i\bar{\y}_i}\T\in \arg\min \lrbr{\Cb\bullet\Yb\colon \redd{\Yb_i}\in\CPP\left(\KK_i\right) \colon \g_i\g_i\T\bullet\redd{\Yb_i} = 1}$ \blue{for some $\redd{\bar{\y}_i}\in\KK_i$, which will be a critical ingredient for our upperbound}. Also, since $\Yb_i$ is feasible for the latter problem if and only if $\left(\red{d_i^2}-2\red{d_i}\f_i\T\x+\f_i\f_i\T\bullet\Xb\right)\redd{\Yb_i}\in\bar{\YY}_i(\x,\Xb)$, we can rescale $\redd{\bar{\y}_i\bar{\y}_i}\T$ to an optimal solution of 
		\begin{align*}
			\min_{\redd{\Yb_i}\in\bar{\YY}_i(\x,\Xb)} \Cb\bullet\redd{\Yb_i} = \Cb\bullet\redd{\bar{\y}_i\bar{\y}_i}\T\left(\red{d_i^2}-2\red{d_i}\f_i\T\x+\f_i\f_i\T\bullet\Xb\right), \mbox{ for any } (\x,\Xb)\in \XX_i.
		\end{align*}
		\blue{The above minimum will be used as an underestimator of $\phi_i(\x,\Xb)$, that is affine over $\XX\subseteq \XX_i$.}
		In addition we get $(\g_i\T\bar{\y}_i)^2 = 1$ and from $\g_i\in\mathrm{int}(\KK_i^*)$ we also get $\g_i\T\bar{\y}_i\geq0$ so that we have $\g_i\T\bar{\y}_i = 1$.	
		
		\red{Let us denote by $\mathrm{val}(\mbox{\ref{eqn:DecomposableSpecialCaseMis1}})$ and $\mathrm{val}\eqref{eqn:DecomposableSpecialCaseMis1Ref}$ the optimal values of the problems  \eqref{eqn:DecomposableSpecialCaseMis1} and \eqref{eqn:DecomposableSpecialCaseMis1Ref} respectively}. In total, we now have 
		\begin{align*}
			\mathrm{val}(\mbox{\ref{eqn:DecomposableSpecialCaseMis1}})
			& \geq
			\inf_{(\x,\Xb)\in\XX}\Ab\bullet\Xb+ \aa\T\x + \sum_{i=1}^{S} \phi_i(\x,\Xb) 
			= \mathrm{val}\eqref{eqn:DecomposableSpecialCaseMis1Ref} \geq
			\inf_{(\x,\Xb)\in\XX}\Ab\bullet\Xb+ \aa\T\x
			+
			\sum_{i=1}^{S} 
			\min_{\Yb_i\in\bar{\YY}_i(\x,\Xb)} \Cb_i\bullet\Yb_i
			\\
			&=
			\inf_{(\x,\Xb)\in\XX}\Ab\bullet\Xb+ \aa\T\x
			+
			\sum_{i=1}^{S}
			\redd{
				\left(
				\Cb_i\bullet\bar{\y}_i\bar{\y}_i\T\left(d_i^2-2d_i\f_i\T\x+\f_i\f_i\T\bullet\Xb\right)
				\right).}
		\end{align*}
		\blue{The final problem is a linear optimization problem over the convex set $\XX$. In case this set is not only bounded (as assumed in ii.) but compact (and the succeeding paragraph will show that it is), the optimum will be attained at an extreme point of $\XX$ which ar of the form $(\x,\x\x\T)$ by assumption iii. In other words, we have}
		\begin{align*}
			&\inf_{(\x,\Xb)\in\XX}\Ab\bullet\Xb+ \aa\T\x
			+
			\sum_{i=1}^{S}
			\redd{
				\left(
				\Cb_i\bullet\bar{\y}_i\bar{\y}_i\T\left(d_i^2-2d_i\f_i\T\x+\f_i\f_i\T\bullet\Xb\right)
				\right)}
			\\
			&= \Ab\bullet\bar{\x}\bar{\x}\T +  \aa\T\bar{\x}
			+
			\sum_{i=1}^{S}
			\redd{
				\left(
				\Cb_i\bullet\bar{\y}_i\bar{\y}_i\T\left(d_i-\f_i\T\bar{\x}\right)^2
				\right) \eqqcolon v_{lower}} 
			,
		\end{align*}
		for some some $(\bar{\x},\bar{\x}\bar{\x}\T)\in\XX$. \blue{We will now show that $\bar{\x}\in\R^{n_x}$ together with $(d_i-\f_i\T\bar{\x})\bar{\y}_i\in\R^{n_y},\ i \in\irg{1}{S}$ give a feasible solution to \eqref{eqn:DecomposableSpecialCaseMis1} whose objective value is also $v_{lower}$, so that the above lower bound also gives an upper bound, thereby closing the relaxation gap (baring the argument for the compactness of $\XX$, which is provided in the sequel)}. 
		
		\blue{
		From the definition of $\XX_i$ and the characterization of $\GG(\FF_i)$ in \eqref{eqn:CharacterizationGGofFFi} we get from $(\bar\x,\bar\x\bar\x\T)\in\XX_i$ that there exists a $\hat{\y}_i\in\R^{n_y}$ (not to be confused with $\bar\y_i$) such that $\f_i\bar\x +\g_i\T\hat{\y}_i = d_i$, and from \cref{prop:LinearPartinKK} we see that actually $\hat{\y}_i\in\KK_i$  so that $d_i-\f_i\bar\x = \g_i\T\hat{\y}_i\geq 0, \ i \in\irg{1}{S},$ since $\g_i\in\KK_i^*,\ i\in\irg{1}{S}$ by assumption \blue{(and these inequalities were the main purpose of introducing $\hat{\y}_i, \ i\in\irg{1}{S}$)}. From \cref{prop:LinearPartinKK}, we also have $\bar{\x}\in\KK_0,$ and we already know that $\bar{\y}_i\in\KK_i, \ i \in\irg{1}{S},$ and since the latter cones are closed and convex we infer that $(d_i-\f_i\T\bar{\x})\bar{\y}_i\in\KK_i,\ i \in\irg{1}{S},$ \blue{as the terms in brackets were shown to be nonnegative}. We then exploit the fact that $\g_i\T\bar{\y}_i = 1$ so that $\f_i\T\bar{\x}+(d_i-\f_i\T\bar{\x})\g_i\T\bar{\y}_i = \f_i\T\bar{\x}+ d_i-\f_i\T\bar{\x} = d_i$. Finally, $(\bar\x,\bar\x\bar\x\T)\in\XX_0\subseteq\XX$ implies that $\f_0\T\bar\x = d_0$, so that $\bar{\x}$ together with $(d_i-\f_i\T\bar{\x})\bar{\y}_i, \ i \in\irg{1}{S},$ give a feasible solution for (\ref{eqn:DecomposableSpecialCaseMis1}) with objective function value equal to $v_{lower}$, so that in total we have 
		$\v_{lower}\geq \mathrm{val}(\mbox{\ref{eqn:DecomposableSpecialCaseMis1}})\geq\mathrm{val}(\mbox{\ref{eqn:DecomposableSpecialCaseMis1Ref}})\geq \mathrm{val}(\mbox{\ref{eqn:DecomposableSpecialCaseMis1}})\geq v_{lower},$
		implying that $\mathrm{val}(\ref{eqn:DecomposableSpecialCaseMis1})=\mathrm{val}(\ref{eqn:DecomposableSpecialCaseMis1Ref})$.
		}
		
		\blue{
			We now supply the argument for the compactness of $\XX$. Since it is bounded by assumption ii.\ it suffices to show that it is closed, which will follow if we can show that $\XX_i, \ i \in\irg{1}{S}$ are closed. So fix any $i\in\irg{1}{S}$.  Throughout the following discussion it is crucial that $\GG(\FF_i)$ is closed and convex, which it is by construction, but also nonempty, which it is by assumption, as laid out in the second paragraph of this proof. All results referenced in the present paragraph (namely \cite[Theorem 9.1]{rockafellar_convex_2015} and \cref{prop:ReccCones}) require $\GG(\FF_i)$ to have all these properties. To start our argument, note that $\XX_i$ is a projection of the set $\GG(\FF_i)$ onto the $(\x,\Xb)$-coordinates. Let $\pi_i\colon \SS^{n_x+n_y+1} \rightarrow \R^{n_x} \times \SS^{n_x}$ be the respective projection map, such that $\pi_i(\GG(\FF_i)) = \XX_i$. We will employ \cite[Theorem 9.1]{rockafellar_convex_2015}, which involves both $\GG(\FF_i)^{\infty}$, the recession cone of $\GG(\FF_i)$, as well as $\mathrm{lin}(\GG(\FF_i))$, the lineality space of $\GG(\FF)$, which is the largest linear subspace contained in its recession cone  $\GG(\FF_i)^{\infty}$. The theorem states that the image $\pi_i(\GG(\FF_i))$ is closed if $\bar{\Xb}\in\GG(\FF_i)^{\infty}$ and $\pi_i(\bar{\Xb}) = (\oo,\Ob)$ imply that $\bar{\Xb} \in\mathrm{lin}(\GG(\FF_i))$. In other words, the intersection of the recession cone of $\GG(\FF_i)$ and the kernel of $\pi_i$ must be contained in the lineality space of $\GG(\FF_i)$ in order to guarantee that said image is closed. Since $\GG(\FF_i)\subseteq \SS^{n_x+n_y+1}_+$ by construction and since $\SS^n_+$ is a closed and convex cone, the same inclusion holds for $\GG(\FF_i)^{\infty}$ by \cref{prop:ReccCones} b) and the definition of recession cones. But the cone of positive semidefinite matrices does not contain a line as it is pointed, hence $\mathrm{lin}(\GG(\FF_i))= \lrbr{\Ob}$. Further, by \cref{prop:ReccCones} e), a description of $\GG(\FF_i)^{\infty}$ can be obtained from \eqref{eqn:CharacterizationGGofFFi} by replacing $d_i$ (and therefore also $d_i^2$) with zero. Now let $\bar{\Xb}\in\GG(\FF_i)^{\infty}$ with $\pi_i(\bar{\Xb}) = (\oo,\Ob)$. We show that this implies $\bar{\Xb} = \Ob\in\mathrm{lin}(\GG(\FF_i))$ by showing that all of its components are vectors and matrices of zeros. Indeed, $\pi_i(\bar{\Xb}) = (\oo,\Ob)$ necessitates that for the $\x$- and $\Xb$-components of $\bar{\Xb}$ we have $\x = \oo$ and $\Xb = \Ob$. But since $\bar{\Xb}$ is positive semidefinite, this forces the component $\Zb_i = \Ob$. Then, the second linear constraint in the description of $\GG(\FF_i)^{\infty}$ reads $\g_i\g_i\T\bullet\Yb_i = 0,$ which forces $\Yb_i = \Ob$ since $\g_i\in\mathrm{int}(\KK_i^*)$ by assumption i., and $\Yb_i\in\CPP(\KK_i)$ by the description of $\GG(\FF_i)^{\infty}$ and \eqref{eqn:SubconesOfCPP}. Consequently, $\y_i = \oo$ again by positive semidefiniteness of $\bar{\Xb}$. Thus, the requirements of \cite[Theorem 9.1]{rockafellar_convex_2015} are met and $\XX_i = \pi_i(\GG(\FF_i))$ is indeed closed.    
		}

		What remains to be shown is that, in case we implemented the changes in the objective function mentioned at the beginning, we can undo these changes in the objective function of the reformulation so that we truly arrive at the reformulation stated in the theorem. 
		So consider a feasible solution to (\ref{eqn:DecomposableSpecialCaseMis1Ref}) after the transformation. We can rearrange the transformed objective function 
		\begin{align*}
			&
			\left(\Ab-0.5\left(\sum_{i=1}^{S}\b_i\f_i\T+\f_i\b_i\T\right)\right)\bullet\Xb 
			+ 
			\left(\aa+\sum_{i=1}^{S}\left(\red{d_i}\b_i-\beta_i\f_i\right)\right)\T\x 
			+ 
			\sum_{i=1}^{S} 
			\redd{
				\left(
				\y_i\T\Cb\y_i+d_i\beta_i
				\right)	}
			,\\
			=& \, 
			\Ab\bullet\Xb + \aa\T\x
			+ \sum_{i=1}^{S}
			\redd{
				\left(
				-0.5\left(\b_i\f_i\T+\f_i\b_i\T\right)\bullet\Xb+\left(d_i\b_i-\beta_i\f_i\right)\T\x+\Cb_i\bullet \Yb_i + d_i\beta_i
				\right)}
			,				
		\end{align*} 
		and consider the $S$ terms in the sum individually. \red{Since $\Xb, \x, \Yb_i, \Zb_i, \y_i$ form members of $\GG(\FF_i),\ i\in\irg{1}{S}$ respectively}, \blue{which either have a decomposition as in (\ref{eqn:ConvexCombinationRepresentationofGGFFi}), where $ \f_i\T\x_l+\g_i\T\y^i_l = d_i,\ l \in \irg{1}{\redd{k}}$, or are the limit points of matrices that have such a decomposition. In the former case, we can rearrange }
		\begin{align*}
			&-0.5\left(\b_i\f_i\T+\f_i\b_i\T\right)\bullet\Xb+\left(\red{d_i}\b_i-\beta_i\f_i\right)\T\x+\Cb_i\bullet \Yb_i + \red{d_i}\beta_i,\\			
			=& \sum_{l=1}^{k}\gl_l \left(-
			\x_l\T\b_i\f_i\T\x_l+\left(\red{d_i}\b_i-\beta_i\f_i\right)\T\x_l+\Cb_i\bullet \y^i_l\left(\y^i_l\right)\T + \red{d_i}\beta_i\right)\\
			=& \sum_{l=1}^{k}\gl_l \left(\x_l\T\b_i\left(\red{d_i}-\f_i\T\x_l\right)  + \beta_i\left(\red{d_i}-\f_i\T\x_l\right)+\Cb_i\bullet \y^i_l\left(\y^i_l\right)\T
			\right)\\
			=& \sum_{l=1}^{k}\gl_l \left(\b_i\g_i\T\bullet\y^i_l\x_l+\Cb_i\bullet \y^i_l\left(\y^i_l\right)\T + \beta_i\g_i\T\y^i_l
			\right)
			= \Bb_i\bullet\Zb_i+\Cb_i\bullet \Yb_i + \cc_i\y_i,
		\end{align*}
		so that we recover the desired form. \blue{In the latter case, we can pass to a limiting argument since linear functions are continuous. }
	\end{proof}

	\section{A new completion theorem}\label{sec:A new completion theorem}
	
	With \cref{thm:DecomposableSpecialCaseMis1} in place, we are now ready to \red{prove} a new completion result \red{via \cref{thm:ZeroGapimpliesComp}} that gives sufficient conditions for a matrix in $\PP\CC\PP_{G_{n\red{+1},1}^S}(\R_+\times\KK\times\R^S_+)$ to be completable.  
	
	\begin{thm}\label{thm:ConditionalCPPCompletion}
		Assume $\KK\subseteq\R^{n}$ is a ground convex cone and that the sets
		\begin{align*}				
			\FF_0\coloneqq\lrbr{\x\in\KK\colon \f_0\T\x = d_0}, \ 
			\FF_i\coloneqq\lrbr{[\x\T,y_i]\T\in\KK\times\R_+\colon \f_i\T\x+g_iy_i = d_i},\ i \in \irg{1}{S}, \ 
		\end{align*}
		are such that
		\begin{enumerate}
			\item[i'.] the vectors $g_i>0,\ i\in\irg{1}{S}$, 
			\item[ii'.] the set $\XX$ (as defined in \cref{thm:DecomposableSpecialCaseMis1}, with $n_x = n$ and $n_y=1$) is bounded, and 
			\item[iii'.]  \redd{the extreme points of $\XX$} are of the form $(\x,\x\x\T).$
		\end{enumerate} 
		\noindent
		Then a partial matrix $\Mb_{G_{n+1,1}^S}\in\PCP_{G_{n+1,1}^S}(\R_+\times \KK\times\R_+^S)$ can be completed to matrix $\Mb\in\CPP\left(\R_+\times \KK \times \R_+^S\right)$ if its fully specified submatrices $\Mb_i\in\SS^{n+2}, \ i\in\irg{1}{S}$ fulfill
		\begin{align*} 
			\Mb_i \coloneqq \begin{bmatrix}
				1 & \x\T &y_i \\
				\x& \Xb  &\z_i\\
				y_i& \z_i\T& Y_i
			\end{bmatrix} &\in \ \CPP(\R_+\times\KK\times\R_+), \quad 
			\begin{array}{l}
				\f_i\T\x + g_i\T y_i = d_i,\\
				\f_i\f_i\T\bullet\Xb + 2\  g_i\f\T\z_i+g_i^2Y_i = d_i^2,			 	
			\end{array}								
		\end{align*}	
		and $\f_0\T\x=d_0,\ \f_0\f_0\T\bullet\Xb=d_0^2$.
	\end{thm}
	\begin{proof}
		Consider the optimization problems 
		\begin{align}\label{eqn:TheoremProblem1}
			\begin{split}
				\inf_{\x,y_1,\dots,y_S} 
				\lrbr{
					\x\T\Ab\x + \aa\T\x
					+ 
					\sum_{i=1}^{S}
					\redd{
						\left(
						\x\T\b_iy_i + C_iy_i^2+c_iy_i
						\right)}
					\colon 
					\begin{array}{l}
						\red{\f_0\T\x = d_0,}\\
						\f_i\T\x + g_i y_i = \red{d_0},\ i \in\irg{1}{S},\\
						\x\in \KK,\
						\y_i\in\R_+,\ i\in\irg{1}{S},				
					\end{array}	
				},				
			\end{split}
		\end{align}
		and 
		\begin{align}\label{eqn:TheoremProblem2}
			\hspace{-1cm}
			\begin{split}
				\inf_{
					\substack{
						\Xb,\x,Y_i,\\ \z_i,y_i
					}
				}
				\lrbr{
					\begin{array}{l}
						\Ab\bullet\Xb + \aa\T\x 
						+ 
						\sum_{i=1}^{S} 
						\redd{
							\left(
							\b_i\T\z_i + C_i Y_i+c_iy_i
							\right)}
						\colon \\
						\begin{array}{rl}
							\red{\f_0\T\x = d_0, }\
							\red{\f_0\f_0\T\bullet\Xb = d_0^2,}  \\
							\f_i\T\x + g_iy_i = \red{d_i}\hspace{0.05cm},
							& i \in\irg{1}{S},
							\\
							\f_i\f_i\T\bullet\Xb + 2  g_i\f_i\T\z_i+g_i^2Y_i = \red{d_i^2},
							& i  \in\irg{1}{S},
							\\
						\end{array}					
					\end{array}		
					\begin{bmatrix}
						1 & \x\T &y_i \\
						\x& \Xb  &\z_i\\
						y_i& \z_i\T& Y_i
					\end{bmatrix} \in \CPP(\R_+\times \KK\times\R_+),\
					i  \in\irg{1}{S} 
				},				
			\end{split}
		\end{align}
		and 
		\begin{align}\label{eqn:TheoremProblem3}
			\hspace{-2cm}
			\inf_{
				\substack{
					\Xb,\x,Y_i,\\ \z_i,y_i
				}
			}
			\lrbr{
				\begin{array}{l}
					\Ab\bullet\Xb + \aa\T\x 
					+ 
					\sum_{i=1}^{S}
					\redd{  
						\left(
						\b_i\T\z_i + C_i Y_i+c_iy_i
						\right)}
					\colon 						
					\\
					\begin{array}{rl}
						\red{\f_0\T\x = d_0, }\
						\red{\f_0\f_0\T\bullet\Xb = d_0^2,}  \\
						\f_i\T\x + g_iy_i = \red{d_i}\hspace{0.05cm},
						& i \in\irg{1}{S},
						\\
						\f_i\f_i\T\bullet\Xb + 2  g_i\f_i\T\z_i+g_i^2Y_i = \red{d_i^2},
						& i  \in\irg{1}{S},
						\\
					\end{array}						
				\end{array}\ 					
				\begin{bmatrix}
					1 & \x\T &y_1 & \dots & y_S \\
					\x& \Xb & \z_1 & \dots & \z_S \\
					y_1&\z_1\T & Y_{1,1} & \dots & Y_{1,S}\\
					\vdots&\vdots&\vdots&\ddots & \vdots \\
					y_S & \z_S\T & Y_{S,1}&\dots & Y_{S,S}
				\end{bmatrix}  \in \CPP\left(\R_+\times \KK \times \R_+^S\right)
			}.				
		\end{align}
		Now \eqref{eqn:chlifting} and \eqref{eqn:characterizationofGGFF} establish $\mathrm{val}\left(\ref{eqn:TheoremProblem3}\right) = \mathrm{val}\left(\ref{eqn:TheoremProblem1}\right)$ and the statement would remain true \red{if} the values $Y_{i,j},\ i\neq j$ were set to be unspecified, \red{$\Yb_{i,i}$ were replaced by $\Yb_{i}$} and $\CPP(\R_+\times\KK\times\R_+^S)$ were replaced by $\CCP_{G_{n,1}^S}(\R_+\times\KK\times\R_+^S)$ due to \eqref{eqn:SparseReformulation}. If we then replaced $\CPP(\R_+\times\KK\times\R_+^S)$ with $\PCP_{G_{n+1,1}^S}(\R_+\times\KK\times\R_+^S)$ we retrieved \red{\eqref{eqn:TheoremProblem2}} due to \cref{prop:CharacterizePCP}. Further,  \cref{thm:DecomposableSpecialCaseMis1} establishes $\mathrm{val}\left(\ref{eqn:TheoremProblem2}\right) = \mathrm{val}\left(\ref{eqn:TheoremProblem1}\right)$ since assumptions i'.-iii'. are sufficient for i.-iii., in particular $g_i>0$ is equivalent \ $g_i\in\mathrm{int}(\R_+^*)$ \redd{(bearing in mind that $\R_+^*=\R_+$)} and we can always write $\b_i = g_i\left(\b_i/g_i\right)$ and $c_i = g_i\left(c_i/g_i\right)$ so that the conditions on $\b_i,c_i,\ i \in \irg{1}{S}$ in \cref{thm:DecomposableSpecialCaseMis1} are always fulfilled. Therefore, the result follows from \cref{thm:ZeroGapimpliesComp}.  
	\end{proof}
	
	In essence the theorem tells us that $\Mb_{G_{n+1,1}^S}$ is completable with respect \redd{to} the matrix cone in question whenever it is feasible for the $\PCP$-relaxation of \eqref{eqn:TheoremProblem3} and said relaxation is uniformly lossless courtesy of the sufficient conditions in \cref{thm:DecomposableSpecialCaseMis1}. By setting $\KK = \R^n_+$ we see that the theorem also covers completion with respect to the classic completely positive matrix cone, and since $G^S_{n+1,1}$ is not block clique in general (see \cref{lem:ArrowHeadMatricesAreChordal}), the result allows us to identify partial matrices that have such a completion and are not recognized as such by existing results (which we summarized in \cref{prop:KnownResults}).
	In addition, more general ground cones are allowed at least for the northwest part of the matrix block. Hence, our theorem contributes to existing results \redd{(presented in \cref{prop:KnownResults})} with respect to the ground cone of the set-completely positive matrix cone as well as with respect to the structure of the specification graph.
	
	\section{Discussion of the sufficient conditions}\label{sec:Sufficient conditions for XX}
	
	As mentioned above, condition\red{s} i.-iii.\ in \cref{thm:DecomposableSpecialCaseMis1} might seem elusive at first, but they allow us to prove a very general version of the theorem \red{and strip its already lengthy proof from the technical details surrounding these assumptions}. We will now make an effort to give more tangible criteria that imply said conditions and discuss each of them separately. Throughout this section, we will frequently use the results on recession cones discussed \cref{prop:ReccCones}.  
	
	For this purpose, we define the sets 
	\begin{align*}		
		\FF_{i}^{\x} 
		&\coloneqq 
		\lrbr{
			\x\in\R^{n_x}
			\colon
			[\x\T,\y_i\T]\T\in \FF_{i},  
			\mbox{ for some } \y_i\in\R^{n_y}
		}, 
		\ i \in\irg{1}{S},
		\quad
		\FF_0^{\x} 
		\coloneqq 
		\FF_0,
		\quad 
		\FF_{\x}
		\coloneqq 
		\cap_{i=0}^S\FF_{i}^{\x},
		\\		
		\XX_{\x} 
		&\coloneqq 
		\lrbr{
			\x\in\R^{n_x}
			\colon 
			(\x,\Xb)\in\XX, 
			\mbox{ for some }\Xb\in\SS^{n_x}
		},
		\\
		\FF_i(\x)
		& \coloneqq 
		\lrbr{
			\y\in\R^{n_y} 
			\colon 
			[\x\T,\y\T]\T\in\FF_i
		}
		=
		\lrbr{
			\y\in\KK_i
			\colon 
			\f_i\T\x+\g_i\T\y = d_i, \ \x\in\KK_0,
		}	
		,
		\ \x\in\R^{n_x}, \  
		\ i \in\irg{1}{S}.
	\end{align*}
	The first four definitions describe projections of $\FF_i, \ i\in\irg{0}{S}, \ \FF$ (the feasible set of \eqref{eqn:DecomposableSpecialCaseMis1}) and $\XX$ \redd{(as defined in \cref{thm:DecomposableSpecialCaseMis1})} onto the $\x$-coordinates respectively. The final sets are projections of slices of the respective set $\FF_i$ \redd{(again, as defined in \cref{thm:DecomposableSpecialCaseMis1})} for a fixed $\x$ onto the $\y_i$-coordinates. 
	
	Our first propositions relates condition i.\ to the boundedness of $\FF_i(.), \ i \in\irg{1}{S}$ and the pointedness of $\KK_i, \ i \in\irg{1}{S}$.  	
	\begin{prop}\label{prop:BoundedFFxx}
		For any $i\in\irg{1}{S}$ we have that, if $\g_i\in\mathrm{int}(\KK_i^*)$ then $\FF_i(\x)$ is bounded for all $\x\in\R^{n_x}$. The converse holds if we also assume that there is an $\bar\x\in\FF_i^{\x}$ such that  $d_i-\f_i\T\bar\x>0$ and $\KK_i$ is not a line. In either case, $\KK_i$ is pointed. 
	\end{prop}
	\begin{proof} 
		Fix any $i\in\irg{1}{S}$ and assume $\g_i\in\mathrm{int}(\KK_i^*)$. We will exploit the fact that $\FF_i(\x) = \lrbr{\y\in\KK_i \colon  \g_i\T\y = d_i-\f_i\T\x}$ is closed and convex since it is the intersection of a closed convex cone and a hyperplane. Thus, whenever it is nonempty, we can show \redd{that} it is still bounded by showing that $\bar\y\in\FF(\x)^{\infty} = \lrbr{\y_i\in\KK\colon \g_i\T\y = 0}$ implies $\bar\y = \oo$, \redd{which suffices due to \cref{prop:ReccCones} a) and e)} and readily follows from $\g_i\in\mathrm{int}(\KK_i^*)= \lrbr{\z\in\R^{n_y}\colon \y\T\z >0,\ \forall \y\in\KK_i\setminus\lrbr{\oo}}$. 
		
		We prove the converse under the additional assumption by considering the contrapositive that $\g_i\notin\mathrm{int}(\KK_i^*)$. \redd{Again employing \cref{prop:ReccCones},} we will construct a \redd{nonzero} element in $\FF_i(\bar\x)^{\infty} = \lrbr{\y\in\KK_i\colon \g_i\T\y = 0}$, where the latter characterization applies since $\FF(\bar\x)\neq\emptyset$ lest $\bar\x\in\FF^{\x}_i$ would  be impossible. Now, since $\g_i\notin\mathrm{int}(\KK_i^*)$ there is an $\y_1\in\KK_i\setminus\lrbr{\oo}\colon \g_i\T\y_1\leq 0$. If $\g_i\T\y_1=0$ we have an element of $\FF_i(\x)^{\infty}$ contradicting boundedness, so say $\g_i\T\y_1<0$.  
		By assumption we have an $\bar\x\in\FF_{i}^{\x}$ such that $d_i-\f_i\T\bar\x>0$ so we have an  $\y_2\in\FF(\bar\x)$ such that $\g_i\T\y_2 = d_i-\f_i\T\bar\x>0$. Then there is a $\gl\in[0,1]$ such that $\y_0\coloneqq \gl\y_1+(1-\gl)\y_2$ gives $\g_i\T\y_0 = 0$ . We know $\y_0\in\KK_i$ by convexity. In case $\y_0 = \oo$, we know that $\KK_i$ contains the line through $\y_1$ and $\y_2$, but is itself not a line by assumption. So pick a $\hat\y\in\KK_i$ not on that line (which also means it cannot be the origin). If $\g_i\T\hat\y = 0,$ we found a nonzero element in the recession cone, and if $\g_i\T\hat\y>0$ or $\g_i\T\hat\y<0,$ we can use $\hat\y$ as a substitute for $\y_1$ or $\y_2$ in the construction of $\y_0$, respectively. In either case after repeating the previous construction, we get $\y_0\in\KK_i\setminus\lrbr{\oo}$ with $\g_i\T\y_0=0$, hence $\y_0\in\FF_i(\x)^{\infty}$ \red{contradicting boundedness}. To prove the final point, it is immediate from $\g_i\in\mathrm{int}(\KK_i^*)$ that the interior of $\KK_i^*$ is nonempty so that the dual cone of $\KK_i^*$ must be pointed, but $\KK_i^{**} = \mathrm{cl}(\mathrm{conv}(\KK_i)) = \KK_i$ since it is a closed convex cone. 				
	\end{proof}
	We conclude that the sufficient conditions in \cref{thm:DecomposableSpecialCaseMis1} necessitate that $\KK_i,\ i \in\irg{1}{S}$ are pointed. This is not too limiting since many important cones, like the positive orthant, the second-order cone, and p-norm cones etc.\ are in fact pointed. Also, for tractable cones, such as the ones just enumerated, we can test \redd{$\g\in\mathrm{int}(\KK^*)$} by solving $\inf_{\y}\lrbr{\g\T\y\colon \aa\T\y = 1\ ,\y\in\KK}$ for some $\aa\in\mathrm{int}(\KK^*)$ (for most cones of interest such an element is known) since that infimum is positive if and only if $\g\in\mathrm{int}(\KK^*)$. Thus, in many interesting cases the condition i.\ can be tested with polynomial-time \redd{work}.

	The next proposition relates the boundedness of $\XX$, i.e.\ assumption ii.\@, to the boundedness of $\FF$ and $\FF_{\x}$. 		
	\begin{prop}\label{prop:BoundednessofXX}
		We have $\FF_{\x} = \XX_{\x}$. Further, for the statements
		\begin{align*}
			\mbox{a) }\XX \mbox{ is bounded}, \quad   \mbox{b) }\FF \mbox{ is bounded}, \quad  \mbox{c) }\FF_{\x} \mbox{ is bounded},
		\end{align*}
		the following holds: 
		\begin{itemize}
			\item We always have that c) is implied by both a) and b).   
			\item Under condition i., we also have that c) implies b) and that a) implies b)
			\item  If we additionally assume that either $\KK_0\subseteq\R_+^{n_x}$ or $\KK_i\subseteq\R_+^{n_y},\ i\in\irg{1}{S}$ then all three are equivalent. 
		\end{itemize}
		Finally, in case at least one $\FF_i, \ i \in\irg{0}{S}$ is bounded then $\FF_{\x},\ \XX$ and $\FF$ are all bounded as well.  
	\end{prop}
	\begin{proof}
		We start by proving $\FF_{\x} = \XX_{\x}$. Let $\x\in\XX_{\x}$ then there is an $\Xb\in\SS_+^n$ so that both together form a matrix in $\GG(\FF_0)$ and there are $(\y_i,\Yb_i,\Zb_i)\in\R^{n_y}\times\SS^{n_y}\times\R^{n_y\times n_x}, \ i \in\irg{1}{S}$ such that all these blocks form matrices in $\GG(\FF_i),\ i\in\irg{1}{S}$ respectively. From \eqref{eqn:characterizationofGGFF} and the definitions of $\FF_i, \ i \in\irg{0}{S}$, we then get $\f_0\T\x = d_0$ and $\f_i\T\x+\g_i\T\y_i = d_i, \ i \in\irg{1}{S}$. In addition, \cref{prop:LinearPartinKK} shows that $\x\in\KK_0,\ \y_i\in\KK_i,\ i \in\irg{1}{S}$ so that $\x\in\FF_0=\FF_0^{\x}$ and $[\x\T,\y_i\T]\T\in\FF_i,$ hence $\x\in\FF^{\x}_i, \ i\in\irg{1}{S}$. Conversely, whenever $\x\in\FF_{\x}$ then there are $\y_i$ such that $[\redd{\x\T},\y_i\T]\T\in\FF_i$. Define $\Xb\coloneqq \x\x\T$ and $\Zb_i\coloneqq \y_i\x\T, \ \Yb_i\coloneqq \y_i\y_i\T, \ i \in\irg{1}{S}$ so that the resulting matrix-blocks are members of $\GG(\FF_i), \ i\in\irg{0}{S}$ respectively and $\x\in\XX_{\x}$ is certified. From this, it also follows that a) implies c) since bounded sets have bounded projections and $\XX_{\x}$ is a projection of $\XX$. 
		
		Similarly, the set $\FF_{\x}$ is bounded if $\FF$ is, since the latter is a projection of the former, hence c) is implied by \redd{b)}, and for the same reason, they can only be empty simultaneously. Conversely, under i.\@,  if $\FF_{\x}$ is bounded, there is an upper bound $M_x$ on the norm of its elements. For $\z\coloneqq [\x\T,\y\T]\T\in\FF$ and an arbitrary $i\in\irg{1}{S}$ we argue the following chain of inequalities
		\begin{align*}
			\hspace{-1cm}
			\|\z\|_2&\leq \|\x\|_2+\|\y\|_2
			\\
			&\leq M_x+ \sup_{\bar\y,\bar\x}\lrbr{\|\bar\y\|_2\colon [\bar\x\T,\bar\y\T]\T\in\FF,\ \bar\x\in\R^{n_x}}
			= 
			M_x+\sup_{\bar\y,\bar\x}\lrbr{\|\bar\y\|_2\colon [\bar\x\T,\bar\y\T]\T\in\FF,\ \bar\x\in\FF_{\x}}
			\\
			&\leq M_x+\sup_{\bar\y,\bar\x}\lrbr{\|\bar\y\|_2\colon [\bar\x\T,\bar\y\T]\T\in\FF,\ \|\bar\x\|_2\leq M_x}\\
			&
			\leq 
			M_x + \sup_{\bar\y,\bar\x}\lrbr{\|\bar\y\|_2\colon  [\bar\x\T,\bar\y\T]\T\in\FF_i,\ \|\bar\x\|_2\leq M_x}
			\\
			&= M_x + \sup_{\bar\y,\bar\x}\lrbr{\|\bar\y\|_2\colon \f_i\T\bar\x+\g_i\T\bar\y = d_i, \ [\bar\x\T,\bar\y\T]\T\in\KK_0\times\KK_i,\ \|\bar\x\|_2\leq M_x}
			\\
			&\leq M_x + \sup_{\bar\y}\lrbr{\|\bar\y\|_2\colon \g_i\T\bar\y \leq \bar{d}\coloneqq  \sup_{\|\bar\x\|_2\leq M_x}\lrbr{d_i-\f_i\T\bar\x}<\infty,\ \bar\y\in\KK_i}\leq M_x+M_y.
		\end{align*}
		The first inequality is a consequence of the triangle inequality and $\z = [\x\T,\oo\T]\T+[\oo\T,\y\T]\T$. The second one is true by boundedness of $\FF_{\x}$ as well as the fact that $(\x,\y)$ is feasible for the supremum and therefore yields a lower bound equal to $\|\y\|_2$. The succeeding equation is true since the additional constraint is redundant for the supremum, and the third and fourth inequalities are true again by boundedness of $\FF_{\x}$ and since $\FF\subseteq \FF_i$. The latter set's definition justifies the next equation. The second-to-last inequality holds since the new supremum is a relaxation of the previous one, in the sense that any previously feasible values for $\bar\y$ remain feasible. But this set is also compact by \cref{prop:BoundedFFxx} since $\g_i\in\mathrm{int}(\KK_i^*)$ and the 2-norm is continuous so that the supremum is attained and thus bounded from above by some $M_y\in\R$, i.e.\ c) implies b) under condition i.

		To see that, under condition i.\@, $\FF$ is bounded whenever $\XX$ is  consider the contrapositive that there is a divergent sequence $\z_k\T\coloneqq [\x_k\T,(\y_1^k)\T,\dots,(\y_S^k)\T]\T\in\FF, \ k \in\N$. If $\lrbr{\x_k}_{k\in\N}$ was a bounded sequence then so would be $\lrbr{\y^k_i}_{k\in\N}, \ i \in\irg{1}{S}$ by \cref{prop:BoundedFFxx} so $\lrbr{\z_k}_{k\in\N}$ could not be divergent. But then $(\x_k,\x_k\x_k\T)\in\XX, \ k \in\N$ is a divergent sequence contradicting boundedness of $\XX$. Hence, a) implies b) given i. 
		
		To prove the equivalence of a), b), and c) under i. and the additional assumptions \redd{on $\KK_i,\ i\in\irg{0}{S}$} it suffices to prove that c) implies a) since b) and \redd{c)} are already equivalent under i. and a) implies c) anyway. 
		So assume that $\FF_{\x}$ is bounded. If it is actually empty, then $\XX$ must be empty, otherwise we could project one of its elements into $\FF_{\x}$. So, assume neither of the two is empty. We will show that the feasible set of \eqref{eqn:DecomposableSpecialCaseMis1Ref}, say $\RR$, is bounded, from which the boundedness of $\XX$ will follow since it is the projection of that feasible set onto the $(\x,\Xb)$-coordinates. By \redd{\cref{prop:ReccCones} e)} the set of directions of recession of \redd{$\RR$} is given by 
		\begin{align*}
			\RR^{\infty}= \lrbr{
				(\Xb,\x,(\Yb_i,\Zb_i,\y_i)_{i\in\irg{1}{S}})\colon 
				\begin{array}{rl}
					\f_0\T\x = 0,\
					\f_0\f_0\T\bullet\Xb = 0,  &  \\
					\f_i\T\x + \g_i\T\y_i = 0,
					& i \in\irg{1}{S},
					\\
					\f_i\f_i\T\bullet\Xb + 2  \g_i\f_i\T\bullet\Zb_i+\g_i\g_i\T\bullet\Yb_i = 0,
					& i  \in\irg{1}{S},\\
					\begin{bmatrix}
						1 & \x\T &\y_i\T \\
						\x& \Xb  &\Zb_i\T\\
						\y_i& \Zb_i& \Yb_i
					\end{bmatrix} \in \CPP(\R_+\times\KK_0\times\KK_i), &i  \in\irg{1}{S},
				\end{array}
			},
		\end{align*}
		while  set of recession directions of $\FF$ is given by,
		\begin{align*}
			\FF^{\infty}=\lrbr{
				(\x,(\y_i)_{i\in\irg{1}{S}})\colon \f_0\T\x = 0,\  \f_i\T\x+\g_i\T\y_i= 0, \ [\x\T,\y_i\T]\T\in\KK_0\times\KK_i, \ i \in\irg{1}{S}
			},
		\end{align*}
		and contains only the origin since under i.\ $\FF$ is bounded whenever $\FF_{\x}$ is bounded \redd{by \cref{prop:BoundedFFxx}}. From this, it is immediate that the $\x$- and $\y_i$-components of elements in $\RR^{\infty}$ must be always zero since the $\CPP$ constraints imply that $\x\in\KK_0,\ \y_i\in\KK_i, \ i \in\irg{1}{S}$ by \cref{prop:LinearPartinKK} so that these components form a member of $\FF^{\infty}$. Thus, we investigate the remaining matrix components and firstly discuss the case $\KK_0\subseteq\R_+^{n_x}$. By \eqref{eqn:SubconesOfCPP} and the definition of completely positive matrix cones, we have 
		\begin{align*}
			\begin{bmatrix}
				\Xb & \Zb_i\T\\ \Zb_i & \Yb_i
			\end{bmatrix} = 
			\sum_{\ell = 1}^{k_i}
			\begin{bmatrix}
				\u^i_{\ell}\\ \v^i_{\ell}
			\end{bmatrix}
			\begin{bmatrix}
				\u^i_{\ell}\\ \v^i_{\ell}
			\end{bmatrix}\T, \quad 
			\begin{bmatrix}
				\u^i_{\ell}\\
				\v^i_{\ell}
			\end{bmatrix}\in\KK_0\times\KK_i\subseteq\R_+^{n_x}\times\KK_i, \quad \ell\in\irg{1}{k_i}, \ i \in\irg{1}{S},
		\end{align*}
		so that the $j$-th row of $\Zb_i\T$, say $\y^i_j$, has $\y^i_j = \sum_{\ell = 1}^{k_i}(\u^i_{\ell})_{j}\v^i_{\ell}\in\KK_i$, since this cone is closed and convex and hence closed under addition and nonnegative scalar multiplication. In addition the rows of $\Xb$, say $\x_j, \ j\in\irg{1}{n_x}$ are elements of $\KK_0$ by an analogous argument. We will now exploit the fact that for a matrix $\Mb\in\SS_+^n$ and a vector $\v\in\R^n$ we have that $\v\T\Mb\v = 0 $ if and only if $\Mb\v = \oo$ where the "only-if"-direction follows from the fact that psd matrices have a decomposition $\Mb= \bar\Mb\T\bar\Mb, \ \bar\Mb\in\R^{r\times n}$, so that $\v\T\Mb\v = \|\bar\Mb\v\|_2^2 =0$ implies $\bar\Mb\v = \oo$ and finally $\Mb\v = \bar\Mb\T\bar\Mb\v = \bar\Mb\oo = \oo$. Now, observe that for all $i\in\irg{1}{S}$ we have the equivalence  between $\f_i\f_i\T\bullet\Xb + 2  \g_i\f_i\T\bullet\Zb_i+\g_i\g_i\T\bullet\Yb_i = 0,$ and
		\begin{align*}
			\begin{bmatrix}
				\f_i \\ \g_i
			\end{bmatrix}\T
			\begin{bmatrix}
				\Xb & \Zb_i\T\\ \Zb_i & \Yb_i
			\end{bmatrix}
			\begin{bmatrix}
				\f_i \\ \g_i
			\end{bmatrix} = 0, 
			\quad \Leftrightarrow \quad 
			\begin{bmatrix}
				\Xb & \Zb_i\T\\ \Zb_i & \Yb_i
			\end{bmatrix}
			\begin{bmatrix}
				\f_i \\ \g_i
			\end{bmatrix} = \oo, 
			\ \mbox{ so that } \
			\f_i\T\x_j+\g_i\T\y^i_j  = 0 ,\ j\in\irg{1}{n_x}.
		\end{align*} 
		Thus, the rows of $\Xb$ and $\Zb_i\T,\ i \in \irg{1}{S}$ constitute members $\FF^{\infty}$ and must therefore be all zero.  Consequently, $0=\g_i\g_i\T\Yb_i= \sum_{\ell = 1}^k(\g_i\T\v^i_{\ell})^2$ which implies $\v^i_{\ell} = \oo$ since $\v^i_{\ell}\in\KK_i$ and $\g_i\in\mathrm{int}(\KK_i^*)$. Hence, $\RR^{\infty}$ contains only the origin and the feasible set of \eqref{eqn:DecomposableSpecialCaseMis1Ref} is bounded \redd{by \cref{prop:ReccCones} a)}, as is its projection $\XX$. In case $\KK_i= \R_+^{n_y}, \ i \in\irg{1}{S}$ we conclude by an analogous argument that $\Yb_i$ and $\Zb_i$ are zero matrices $i \in\irg{1}{S}$ and that $\f_i\f_i\T\bullet\Xb = 0, \ i \in\irg{0}{S}$ which together with $\CPP(\KK_0)\ni \Xb = \sum_{j = 1}^{k}\x_j\x_j\T, \ \x_j\in\KK_0, \ j \in\irg{1}{k}$  implies that $\f_i\T\x_j = 0, \ i\in\irg{0}{S},\ j\in\irg{1}{k}$ so that these components of $\Xb$ together with $\y_i = \oo\in\KK_i,\ i \in\irg{1}{S}$ form members of $\FF^{\infty}$ so that they must all be zero as well.  
		
		To prove the final statement, assume that $\FF_i$ is bounded for an arbitrary $i\in\irg{0}{S}$ and observe that the boundedness of $\FF_{\x}$ and $\FF$ follows from their definitions and the fact that bounded sets have bounded projections. Also in case $\FF_i$ is bounded it is also compact so that $\GG(\FF_i)$ \redd{is} compact as well by \cite[Proposition 4.]{gabl_finding_2024}. But then $\XX_i$ is bounded as it is a projection of $\GG(\FF_i),$ which finally renders $\XX\subseteq \XX_i$ \redd{to be} bounded as well.  			
	\end{proof}
	We, firstly, \redd{would} like to remark that the assumption $\KK_i\subseteq\R_+^{n_y}, \ i \in\irg{1}{S}$ is readily fulfilled by the instances of $\KK_i$ considered in \cref{thm:ConditionalCPPCompletion} where they are all equal to $\R_+$. So\redd{,} at least in this important case, we can fully understand the implications of, and the conditions for the boundedness assumption on $\XX$. Note, that one way to guarantee that a particular $\FF_i$ is bounded is to have $\f_i\in\mathrm{int}(\KK_0^*)$ and $d_i\geq 0$ by \cref{prop:BoundedFFxx} (use the first statement in that proposition for the case $\KK_0=\lrbr{\oo}$), which again can be checked easily.  
	
	Secondly, we like to briefly comment on the difficulties that arise in case neither $\KK_0$ nor $\KK_i, i \in\irg{1}{S}$ are subcones of the positive orthant and $\FF_i,\ i\in\irg{0}{S}$ are individually unbounded. In this case, problems arise when trying to show that $\XX$ is bounded if $\redd{\FF_{\x}}$ (or equivalently $\FF$, given i.\ holds) is. \redd{Firstly}, note that if $\FF$ is bounded then so is $\GG(\FF)$ (see e.g. \cite[Proposition 4.]{gabl_finding_2024}), however $\RR$ (as defined in \cref{prop:BoundednessofXX}) may not be a projection of that set (otherwise we could always close the relaxation gap, which may still be the case but no proof currently exists) but merely contains that projection so that it could still be unbounded and the same is true for $\XX$, the projection on $\RR$ onto the $(\x,\Xb)$-coordinates. However, the $\x$- and $\y_i$-coordinate\redd{s} of $\RR^{\infty}$ would still be forced to zero since $\FF^{\infty}$ still contained only the origin. But the problematic player is in fact $\Xb$, since the other matrix blocks would be bounded once $\Xb$ is bounded by elementary arguments, which we skip here for brevity's sake. Now, if $(\x,\Xb)\in\XX$ then $\Xb= \sum_{l=1}^{k_i}\gl_l^i\x_l^i(\x_l^i)\T,\  \x_l^i\in\FF_i^{\x},\ l\in\irg{1}{k_i}$ for all $i \in\irg{0}{S}$ following the decomposition in \eqref{eqn:ConvexCombinationRepresentationofGGFFi}. However, any $\x_l^i$ is not necessarily a member of $\FF_j^{\x}, \ j\neq i$ so that these components are not necessarily bounded. If for any feasible $\Xb$ we had $S$ decompositions that are coordinated such that $\x_l^i\in\FF_j^{\x}, \ j\in\irg{0}{S}$ for all $l\in\irg{1}{k_i}, \ i \in\irg{0}{S}$ then boundedness of $\Xb$ would immediately follow from boundedness of $\FF$. But there is no guarantee that such a coordinated decomposition exists. Interestingly, this mirrors a sufficient condition for completability described in \cite[Theorem 9]{gabl_sparse_2023}.

	Finally, we will now investigate sufficient conditions for iii., arguably the most difficult assumption of the three. As a consequence of this difficulty, these conditions will be quite strict, but, as we will argue in \cref{sec:Application}, they are flexible enough to allow for an interesting application of our theory. We present our findings in two steps. Firstly, \cref{prop:SufficientConditionForXXassumption} establishes a sufficient condition on $\FF_{i}^{\x}, \ i \in\irg{0}{S}$ that implies iii. Secondly, \cref{prop:CharacterizationOfsufficientConditionOnXXi} shows how these conditions can be checked by investigating the respective vectors $\f_i, \ i \in\irg{0}{S}$. 
	
	\begin{prop}\label{prop:SufficientConditionForXXassumption}
		Under condition i.\@, if there is an $i^*\in \irg{\red{0}}{S}$ such that $\FF_{i^*}^{\x}\subseteq\FF_{\red{i}}^{\x}$ for all $\red{i} \in\irg{\red{0}}{S}$ then the extreme points of  $\XX$ are of the form $(\x,\x\x\T)$. 
	\end{prop}
	\begin{proof}
		We will show that the assumption of the proposition implies that $\XX_{i^*}\subseteq \XX_{\red{i}}$ for all $\red{i}\in\irg{\red{0}}{S}$. From this the statement will follow, since then $\XX = \bigcap_{i=0}^S\XX_i = \XX_{i^*}$ whose extreme points \red{are} of the form $(\x,\x\x\T)$. This is the case because the sets $\XX_i$ are the projections of $\GG(\FF_i)$ onto the $(\x,\Xb)$ coordinates, and the extreme points of a projection of a convex set correspond to extreme points of the projected set. So let $\left(\x,\Xb\right)\in \XX_{i^*}$. If $i^*\neq 0$ there exist $\Yb,\ \y,\ \Zb$ that complete $(\x,\Xb)$ to an element of $\GG\left(\FF_{i^*}\right)$. \redd{We distinguish two cases. Either, we have }
		\begin{align}\label{eqn:Decomp}
			\begin{bmatrix}
				1 & \x\T &\y\T \\
				\x& \Xb  &\Zb\T\\
				\y& \Zb& \Yb
			\end{bmatrix} = \sum_{\red{l}=1}^{k} \gl_l \begin{bmatrix}
				1\\\x_\red{l}\\\y_\red{l}
			\end{bmatrix}
			\begin{bmatrix}
				1\\\x_\red{l}\\\y_\red{l}
			\end{bmatrix}\T \colon \left[\x_\red{l}\T,\y_\red{l}\T\right]\T \in \FF_{i^*},\ \red{l\in\irg{1}{k},}
		\end{align}
		for some $\ \gl_l\geq 0,\  l\in\irg{1}{k},\ \sum_{\red{l=1}}^{k}\gl_l = 1,$ and if $i^* = 0$ then the north west $(n+1)\times(n+1)$ part of the above equation holds with $\x_l\in\FF_{i^*} = \FF_0, \ l\in\irg{1}{k}$. In any case, our assumption implies that each $\x_l\in\FF^{\x}_i, \ i\in\irg{0}{S}, \ l\in\irg{1}{k}$ so that there are $\red{\y^i_l\in\R^{n_y}}$ such that $[\x_\red{l}\T,(\red{\y^i_l})\T]\T \in \FF_{i}$ for any $l\in\irg{1}{k},\ i\in\irg{1}{S}$. If we replace $\y_l$ with the respective $\red{\y^i_l}$ in \eqref{eqn:Decomp} we obtain new values for $\Yb,\ \y,\ \Zb$ that certify $\left(\x,\Xb\right)\in \XX_{i}$. \redd{In the second case,} $(\x,\Xb)$ are such that they can only be completed to a limit point of $\GG(\FF_{i^*})$. Then \redd{the members of} the respective converging sequence can be turned into a sequence in $\GG(\FF_i), \ i \in\irg{1}{S}$ in an analogous way. Convergence of the respective $\y$-components of the sequence is then guaranteed by the convergence of the $\x$-\redd{components}, which implies that the sequence of these components is bounded so that by \cref{prop:BoundedFFxx} the sequence of the $\y$-components is bounded as well. This concludes the proof. 
	\end{proof}
	The above proposition gives a sufficient condition in terms of projections of $\FF_i,\ i \in \irg{1}{S}$. For the type of $\FF_i$ considered in \cref{thm:DecomposableSpecialCaseMis1}, this sufficient condition \red{can be tested for as shown in the sequel}. \redd{But first, we will prove a lemma which will be useful in the present, but also in the final section.}
	
	\begin{lem}\label{lem:SetImplications}
		Let $\KK\subseteq \R^n$ be a \red{ground} cone and $\R\supseteq \KK_y\in\lrbr{\lrbr{0},\R_+}$, so that $\KK_y^*\in\lrbr{\R,\R_+}$ respectively. Consider the two statements:
		\begin{enumerate}
			\item[a)] 
			$\lrbr{\x\in\KK\colon \f_1\T\x+y = d_1,\ y\in\KK_y}\subseteq  
			\lrbr{\x\in\KK\colon \f_2\T\x\leq d_2}$. 
			\item[b)]  $\exists\gl\in\KK_y^*\colon \gl d_1\leq d_2, \ \gl\f_1-\f_2\in\KK^*$.
		\end{enumerate} 
		Then b) implies a). And if either 
		\begin{enumerate}
			\item[1)] $\KK$ is a polyhedral cone or, 
			\item[2)] the first set in a) contains a point $\bar\x\in\mathrm{ri}(\KK)$,
		\end{enumerate}
		then both are equivalent. 
	\end{lem}
	\begin{proof}
		We argue the following chain of implications: 
		\begin{align*}
			& \lrbr{\x\in\KK\colon \f_1\T\x+y = d_1,\ y\in\KK_y}\subseteq  
			\lrbr{\x\in\KK\colon \f_2\T\x\leq d_2}\\
			\Leftrightarrow
			&\sup_{\x\in\R^n}
			\lrbr{
				\f_2\T\x \colon \f_1\T\x+y = d_1,\ \x\in\KK,\ y\in\KK_y
			}\leq d_2\\
			\Leftarrow
			&
			\inf_{\gl\in\R}
			\lrbr{
				\gl d_1
				\colon 
				\gl \f_1-\f_2 \in\KK^*, \ \gl \in\KK_y^*
			}\leq d_2\\
			\Leftarrow
			& 
			\exists \gl\in\KK_y^*\colon \gl d_1\leq d_2, \ \gl\f_1-\f_2\in\KK^*.
		\end{align*}
		The first equivalence is immediate even if either of the two sets is empty, since then, the first set is empty and the supremum evaluates to $-\infty$ if and only if the feasible set is empty. The next implication is a consequence of weak conic duality, and the second implication is a consequence of the definition of an infimum. Both 1) and 2) guarantee that the duality gap is zero and that the dual attains its optimal value. Thus, the last two implications also hold in the other direction, respectively. 
	\end{proof}

	\begin{prop}\label{prop:CharacterizationOfsufficientConditionOnXXi}
		Assume i.\ holds. For $i^*\in \irg{\red{0}}{S}$ we have that  $\FF_{i^*}^{\x}\subseteq\FF_{\red{i}}^{\x}$ holds for all $\red{i} \in\irg{\red{0}}{S}$ if one of the following conditions hold: 
		\begin{enumerate}
			\item[a)] The index $i^* = 0$ and for all $i\in\irg{1}{S}$ there is a $\gl_i\in\R $ such that $\gl_i d_0\leq d_i,\  \gl_i\f_0-\f_i\in\KK_0^*$. 
			\item[b)] The index $i^* \neq 0$, the set $\FF_0=\FF_0^{\x} = \R^{n_x}$ (i.e.\ iff $\f_0 = \oo,\ d_0 = 0$) and for all $i\in\irg{1}{S}$ there is a $\gl_i\in\R_+$ such that $\gl_i d_{i^*}\leq d_i,\  \gl_i\f_{i^*}-\f_i\in\KK_0^*$.
		\end{enumerate}
	\end{prop}
	\begin{proof}
		Note, that due to $\g_i\in\mathrm{int}(\KK_i^*),$ we have $\FF_i^{\x} = \lrbr{\x\in\KK_0 \colon \f_i\T\x \leq d_i},\ i \in\irg{1}{S}$. For the sufficiency of a) we apply \cref{lem:SetImplications} with $\KK_y = \lrbr{\redd{0}}$ so that $\FF_0$ takes the role of the left-hand side set in a) of that lemma, while the role of the right-hand side set is taken by $\FF_i^{\x}, \ i \in\irg{1}{S}$ consecutively. For the sufficiency of b) we have that $\FF_{i^*}^{\x}\subseteq\FF_0^{\x}= \R^{n_x}$ trivially holds. The other containments are proved again via \cref{lem:SetImplications} using the case $\KK_y = \R_+$, after introducing slack variables.    
	\end{proof}		
	Either of the two conditions above can be checked by solving $S$ conic feasibility problems, which is of polynomial time \redd{work} in case $\KK_0$ is a tractable cone, like the positive orthant or the second order cone. Condition b) is trivially fulfilled if for example all $[d_i,\f_i\T]\T$, and hence all $\XX_i$, are identical since $\oo\in\KK^*$ for any convex cone $\KK$. If the $i^*$-th of these vectors is in $\R_+\times\KK_0^*$ then it suffices that they are identical up to a positive scalar. The former case was used in \cite{bomze_two-stage_2022} to reformulate the Two-Stage Stochastic Standard Quadratic Problem, but we see that our theory is much more versatile than this. 
	
	We close the discussion with three examples. The first one shows that the partial matrix discussed in \cref{exmp:noncompletableExample} is not a counter-example to our theorem. The second one shows that if all fully specified submatrices have their rank equal to one, the completion is trivial, but also certified by our theorem. The final example gives a partial matrix that lacks the rank one property but is still completable and successfully identified as such by our new sufficient conditions.

	\begin{exmp}
		An interesting exercise is to test whether there is a choice for $\FF_i,\ i = 1 ,2$ that fulfills the requirements of \cref{thm:ConditionalCPPCompletion} and would certify the compatibility of the partial matrix from \cref{exmp:noncompletableExample}, with respect to $\CPP(\R^4_+)$. This would, of course, invalidate our theorem, but we will now demonstrate that this \redd{does not} happen. Scaling the matrices so that the northwest entry becomes $1$, we inspect the matrices 
		\begin{align*}
			\begin{bmatrix}
				1 &\sfrac{1}{2} & 0 & \sfrac{1}{2} \\
				\sfrac{1}{2} & 1 & \sfrac{1}{2} & 0\\
				0 & \sfrac{1}{2} & \sfrac{1}{3} & *\\
				\sfrac{1}{2} & 0 &* & \sfrac{1}{3}
			\end{bmatrix}, \ 
			\begin{bmatrix}
				1 &\sfrac{1}{2} & 0  \\
				\sfrac{1}{2} & 1 & \sfrac{1}{2} \\
				0 & \sfrac{1}{2} & \sfrac{1}{3} \\
			\end{bmatrix},\ 
			\begin{bmatrix}
				1 &\sfrac{1}{2}  & \sfrac{1}{2} \\
				\sfrac{1}{2}  & 1 & 0\\
				\sfrac{1}{2} & 0 & \sfrac{1}{3}
			\end{bmatrix}. 
		\end{align*}  
		Note that $\KK_i = \R_+,\ i = 0,1,2$ in this example and we seek values for $f_i,g_i,\ i = 1,2$ \redd{such that $g_i>0, \ i = 1,2$} and the condition\redd{s} on $\XX_i,\ i = 1,2$ holds and the matrices fulfill the respective linear constraints. The first of these requirements is already impossible for $\FF_1$, since we get 
		$\redd{\tfrac{1}{2}}f_1 + 0 g_1 = 1$ implies that $ f_1 = 2$  and the second constraint $f_1^2 +f_1g_1  + \redd{\tfrac{1}{3}}g_1^2 = 1 $ yields $4 + 2g_1 +\redd{\tfrac{1}{3}}g_1^2 = 1$, which requires $2g_1+\tfrac{1}{3}g_1^2<0$, hence $g_1<0$, which contradicts $g_i\in\mathrm{int}(\R_+^*\redd{)}$, i.e. $g_i>0, \ i =1,2$. A similar argument applies to $\FF_2$. 
	\end{exmp}
	
	\begin{exmp}\label{exmp:RankOne}
		As a consequence of \cite[Theorem 9]{gabl_sparse_2023}, matrices in $\PCP$ can be completed if all the fully specified submatrices have their rank equal to one. To see this \redd{let $G \coloneqq G^S_{n,1}$ for ease of notation} and consider any $\Mb_G\in\PCP_{\redd{G}}(\R_+\times\KK\times\R_+^S)$ with $(\Mb_G)_{1,1} = 1$, perhaps after a scaling, with fully specified submatrices given by $\Mb_i\in\CPP(\R_+\times\KK \times\R_+), \ i \in\irg{1}{S}$. If $\Mb_i$ have rank one then  $\Mb_i=  \z_i\z_i\T$ for some $\z_i= [1,\ \x_i\T,\ y_i ]\T\in\R_+\times\KK\times\R_+$. Since the northwest blocks are identical for all $i\in\irg{1}{S}$ we have for any $j\neq i$  that $\x_i\x_i\T= \x_j\x_j\T$, which implies that $\x_i$ is equal to either $\x_{j}$ or $-\x_j$, but since $\KK$ is pointed the latter case can be excluded so that $\x_i = \x\in\KK, \ i \in\irg{1}{S}$. In other words, we have
		\begin{align*}
			\Mb_i \coloneqq \begin{bmatrix}
				1\\\x\\y_i
			\end{bmatrix}
			\begin{bmatrix}
				1\\\x\\y_i
			\end{bmatrix}\T = 
			\begin{bmatrix}
				1 & \x\T  & y_i\\
				\x & \x\x\T& y_i\x\\	 		
				y_i & y_i\x& y_i^2
			\end{bmatrix}\ 
			\mbox{such that }
			\begin{bmatrix}
				1\\\x\\y_i
			\end{bmatrix}\in \R_+\times\red{\KK}\times\R_+, \ i \in \irg{1}{S}.
		\end{align*}  
		But then, a completion of $\Mb_G$ to a matrix in $\CPP(\R_+\times\KK\times\R_+^S)$ is trivially given by $\z\z\T$ where $\z = [1,\x\T,y_1,\dots,y_S]\T$. 
		Our theorem is able to correctly \red{certify} this completability unless $\x = \oo,$ and $y_i = 0$ for one $i\in\irg{1}{S}$. 
		Indeed, one can choose any $[\f\T,g]\T\in \mathrm{int}\left(\red{\KK^*}\times\R_+\right)$ and get $\f\T\x+gy_i >0,\ i \in\irg{1}{S}$, so that we can set $[\f_i\T,g_i]\T = [\f\T,g]\T/(\f\T\x+gy_i)\red{,\ d_i = 1},\  i \in\irg{1}{S}$ \red{and also $\f_0=\oo,\ d_0 = 0$}. Clearly, the resulting $\FF_i, \ i \in \irg{0}{S}$ are bounded, hence $\XX$ is bounded by \cref{prop:BoundednessofXX}, and $\f_i,\ d_i \ i \in\irg{1}{S}$ fulfill condition b) from \cref{prop:CharacterizationOfsufficientConditionOnXXi} (just relabel $i \in\irg{1}{S}$ in order of increasing factors $\f\T\x+gy_i$ and set $i^* = 1$ and $\gl_i = 1, \ i\in\irg{1}{S}$) so that the condition\red{s} on $\XX$ hold. By construction we have $\red{\f_i}\T\x+g_iy_i = 1$ which also provides 
		\begin{align*}
			\f_i\f_i\T\bullet\x\x\T+2g_i\f_i\T\x y_i+ g_i^2y_i^2 = \left(\red{\f_i}\T\x+g_iy_i\right)^2 = 1.
		\end{align*}
		In total, all the requirements in \cref{thm:ConditionalCPPCompletion} are fulfilled and $\Mb_G$ is correctly identified as being completable.  
	\end{exmp}

	\begin{exmp}
		A more tangible example is given by the partial matrix and respective principal submatrices
		\begin{align*}
			\hspace{-1cm}
			\Mb_{G_{2,1}^1} \coloneqq  
			\begin{bmatrix}
				1 & 0.45 & 0.55 & 0.275\\
				0.45 & 0.3 & 0.15& 0.025\\
				0.55 & 0.15 & 0.4& *\\
				0.275 & 0.025 &  * & 0.6
			\end{bmatrix},\
			\Mb_1 \coloneqq  
			\begin{bmatrix}
				1 & 0.45 & 0.55\\
				0.45 & 0.3 & 0.15 \\
				0.55 & 0.15 & 0.4
			\end{bmatrix},\ 
			\Mb_2 \coloneqq
			\begin{bmatrix}
				1 & 0.45 & 0.275\\
				0.45 & 0.3 & 0.025 \\
				0.275 & 0.025 & 0.6
			\end{bmatrix} 		
			.
		\end{align*}	 
		Both $\Mb_1,\Mb_2 \in \CPP(\R_+^3)$, since both are nonnegative, positive semidefinite, and their dimension is \red{not bigger than} 4. Further, both have rank equal to 2, so we are not in the special case discussed before. Now, the conditions on $\FF_i$ can be fulfilled by choosing $f_1 = f_2 = g_1 = 1,$ and $g_2=2$ so that the conditions on $\XX$ trivially hold, which certifies that there must be a completion $\Mb\in \CPP(\R^4_+)$ of  
		$\Mb_{G_{2,1}^1}$. In fact, such a completion is achieved by specifying the missing entry as 0.25. 
	\end{exmp}
	
	\section{Application to sparse reformulations of inequality constrained conic QPs}\label{sec:Application}
	In this section, we will derive an ex-post certificate for the tightness of a sparse relaxation of an inequality-constrained conic quadratic optimization problem given by:
	\begin{align}\label{eqn:QPinequality}
		\inf_{\x\in\KK}\lrbr{\x\T\Ab\x+2\aa\T\x \colon \Fb\x \leq\d},
	\end{align}
	where $\Ab\in\SS^n,\ \aa\in\R^n,\ \Fb\in\R^{m\times n}, \ \d\in\R^m$ and $\KK\subseteq\R^n$ is a \red{ground} cone. By \eqref{eqn:characterizationofGGFF} (see \cite[Theorem 1.]{burer_copositive_2012}), this problem can be exactly reformulated into a copositive optimization problem after introducing slack variables $\y\in\R_+^m$ to get an equality constraint problem.  We then have the exact reformulation: 
	\begin{align*}		
		&
		\inf_{
			\x\in\KK,\ \y\in\R_+^m
		}
		\lrbr{
			\x\T\Ab\x+\aa\T\x 
			\colon 
			\Fb\x + \y=\d}\\		
		=&		
		\inf_{
			\x,\y,\Xb,\Zb,\Yb
		}
		\lrbr{
			\Ab\bullet\Xb+\aa\T\x 
			\colon 
			\begin{array}{l}
				\Fb\x + \y=\d,\\
				\diag(\Fb\Xb\Fb\T)+2\diag(\Zb\Fb\T)+\diag(\Yb) = \d\circ\d,\\
				\begin{bmatrix}
					1 & \x\T & \y\T \\
					\x & \Xb & \Zb\T\\
					\y & \Zb & \Yb
				\end{bmatrix}\in\CPP(\R_+\times\KK\times\R_+^m)
				\subseteq
				\SS^{n+m+1} ,
			\end{array}				
		}.		
	\end{align*}
	The fact that we have to introduce slack variables increases the order of the conic constraints by $m$. This is a significant complication since testing membership in the set-completely positive cone is an NP-hard problem, which in practice also scales poorly with the order of the tested matrix, and so do approximation schemes for this matrix cone.
	
	However, after introducing the slack variables, the QP has the same form as \eqref{eqn:DecomposableSpecialCaseMis1} so that we can in principle construct the respective sparse relaxation as in \eqref{eqn:DecomposableSpecialCaseMis1Ref} (which would essentially look like \eqref{eqn:TheoremProblem2} with $S=m$, $g_i=1,\ \b_i=\oo,\ C_i=0,\  c_i=0, \ i\in\irg{1}{m}$ the vectors $\f_i,\ i \in\irg{1}{m}$ given by the rows of $\Fb$ and the right-hand sides fixed to $(\d)_i$ and $(\d)_i^2, \ i \in\irg{1}{m}$ respectively an\redd{d} $y_i$ are the entries of $\y$), and even close the relaxation gap in case we find the conditions of \cref{thm:DecomposableSpecialCaseMis1} fulfilled, where only the conditions ii.\ and iii.\ are even in question, since i.\ and the condition on the objective function coefficients hold trivially, as $\y$ is absent from the objective and $y_i, \ i \in\irg{1}{m}$ have constraint coefficients equal to $1\in\mathrm{int}(\redd{\R_+^*})$. The resulting relaxation involves matrix constraints of smaller order, thus alleviating computational burden, especially in cases where $m$ is large. 
	
	Even in cases where the relaxation gap cannot be closed based on \cref{thm:DecomposableSpecialCaseMis1}, one may still find the true solution by solving the relaxation, and there are two immediate ways to tell whether one was successful in finding such a solution. The first one is to check the rank of the matrix blocks of the optimal solution of the relaxation. If all of them are of rank one, then a similar argument as in \cref{exmp:RankOne} yields a completion of the matrix blocks to a feasible solution of the exact reformulation with the same optimal value as the relaxation, so that optimality is certified. A second certificate of optimality can be obtained from the $\x$-part of the optimal solution of the relaxation. By \cref{prop:LinearPartinKK}, it is feasible for the original problem, so that it yields an upper bound on that problem. In case this upper bound is identical to the lower bound obtained from solving the relaxation, we again conclude that the relaxation was exact.  
	
	Such optimality certificates are important, for example in global optimization schemes like branch and bound, where the feasible region is dissected into smaller areas, which are iteratively generated and discarded based on the information gathered from upper and lower bounds of restrictions of the optimization problem to those regions (\redd{lower and upper} bounds are typically generated via relaxations such as \eqref{eqn:DecomposableSpecialCaseMis1Ref} and local optimization procedures respectively). If the optimal solution of such a subregion is found, no further dissection of that region is necessary, and computational resources may be diverted to solving other subproblems. Thus, the sooner optimality is certified, the better. We will now discuss a new certificate of optimality for the sparse reformulation of \eqref{eqn:QPinequality} based on \cref{thm:ConditionalCPPCompletion}, that can be checked with reasonable effort. 
	\begin{lem}\label{lem:EquivalentLinear}
		Let 
		\begin{align*}
			\begin{bmatrix}
				1 & \x\T & \y\T\\
				\x& \Xb  & \Zb\T\\
				\y& \Zb  & \Yb
			\end{bmatrix}\in\SS^{n_x+n_y+1}_+,\ \aa\in\R^{n_x}, \ \b\in\R^{n_y}, \ r\in\R,
		\end{align*}
		then the following two conditions are equivalent
		\begin{align*}			
			\hspace{-1cm}
			\begin{array}{r}			
				\aa\T\x+\b\T\y = r\ ,\\
				\aa\aa\T\bullet\Xb+2\aa\b\T\bullet\Zb+\b\b\T\bullet\Yb = r^2,
			\end{array}
			\ 
			\Leftrightarrow
			\
			\aa\aa\T\bullet\Xb+2\aa\b\T\bullet\Zb+\b\b\T\bullet\Yb-2r\aa\T\x-2r\b\T\y+r^2 = 0	
		\end{align*}
	\end{lem}
	\begin{proof}
		Follows directly from \cite[Proposition 3]{burer_copositive_2012} as a special case.  
	\end{proof}
	
	\begin{thm}\label{thm:OptimalityConditions}
		Let $\Xb\in\SS^{n},\ \x\in\R^{n},\ \z_i,\in\R^{n},\ y_i\in\R,\ Y_{i}\in\R, \ i \in\irg{1}{S}$ be an 
		optimal solution to the sparse relaxation of \eqref{eqn:QPinequality} and denote by $\Mb_i\in\SS_+^{n+1}$ the $i$-th matrix block in the sparse reformulation (comprised of $\Xb,\ \x,\ \redd{\z_i,\ Y_i}$ and $\redd{y_i}$) for $i\in\irg{1}{m}$ respectively.
		Then the optimal value of the relaxation is the optimal value of \eqref{eqn:QPinequality} if one of the following conditions holds:
		\begin{enumerate}
			
			\item[a)] There is a vector $\u\in\mathrm{int}(\KK^*)$ and positive numbers $w_i>0,\ \ga_i>0,\ i\in\irg{1}{m}$ such that for the vectors $\v_i\coloneqq [-1,\ga_i\u\T,w_i]\T\in\R^{n+2}$ we have $\Mb_i\v_i = \oo,\ i\in\irg{1}{S}$ respectively.
			
			\item[b)] There is a vector $\v\coloneqq[-1,\u\T]\T\in\R^{n+1}$ in the kernel of the \redd{northwest} matrix block formed by $(\x,\Xb)$ with $\u\in\mathrm{int}(\KK^*)$ such that $\u\T\x=1,\ \x\in\KK$ implies $\Ab\x\leq\d$. 
			
		\end{enumerate}		
	\end{thm}
	\begin{proof}
		To prove sufficiency of a) consider that for every $i \in\irg{1}{S}$ we have  
		\begin{align}\label{eqn:RedundantConstraints}
			\v_i\T\Mb_i\v_i = 
			\ga_i^2\u\u\T\bullet\Xb+2w_i\ga_i \u\T\z_i+w_i^2Y_i-2\ga_i\u\T\x-2w_i y_i+1 = 0,
		\end{align}
		which by the positive definiteness of $\Mb_i$ and \cref{lem:EquivalentLinear} is equivalent to 
		\begin{align*}
			\ga_i\u\T\x+w_i y_i =& 1,\\
			\ga_i^2\u\u\T\bullet\Xb+2w_i\ga_i\u\T\z_i+w_i^2 Y_i =& 1. 
		\end{align*}
		Now define $\u_i\coloneqq \ga_i\u\in\mathrm{int}(\KK^*)$ and  $\FF_i\coloneqq\lrbr{(\x\T,y_i)\T\in\KK\times\R_+\colon \u_i\T\x+w_iy_i = 1 }, \ i \in\irg{1}{m}, \ \FF_0\coloneqq\lrbr{\x\in\R^n\colon \oo\T\x = 0} = \R^n$ along with the respective instances of $\XX_i, \ i \in\irg{0}{m}$. We will show that these constructions fulfill the requirements of \cref{thm:ConditionalCPPCompletion}. Firstly, condition i.\ holds since $w_i>0,\ i \in\irg{1}{m}$, while condition ii.\ holds since $\u_{i}\in\mathrm{int}(\KK^*), \ i \in\irg{1}{m}$ so that $\FF_{i}, \ i \in\irg{1}{m}$ are all bounded by an argument analogous to that for the "if"-direction of \cref{prop:BoundedFFxx}, which entails that $\XX\coloneqq \cap_{i=0}^m\XX_i$ is bounded by \cref{prop:BoundednessofXX}. In addition, setting $i^*\coloneqq \arg\max_{i}\lrbr{\ga_i}$, we have $\u_{i^*}-\u_{i} = (\ga_{i^*}-\ga_i) \u \in\KK^*$ since $\ga_{i^*}-\ga_i\geq 0$ for all $i\in\irg{1}{m}$ so that by Propositions \ref{prop:SufficientConditionForXXassumption} and \ref{prop:CharacterizationOfsufficientConditionOnXXi}b) (setting $\gl_i=1, \ i \in\irg{1}{m}$) the set $\XX$ fulfills condition iii.\ of \cref{thm:ConditionalCPPCompletion}.		
		Thus, the matrix blocks of the sparse relaxation can be completed to a matrix in $\CPP(\R_+\times\KK\times\R_+^m)$ by \cref{thm:ConditionalCPPCompletion}, which is in the feasible set of the exact reformulation and yields the same optimal value as the sparse reformulation since the objective only involves $\Xb$ and $\x$, hence the relaxation was exact. 
		
		Similarly, under b) we have 
		\begin{align*}
			\begin{bmatrix}
				-1\\\u
			\end{bmatrix}\T
			\begin{bmatrix}
				1 & \x\T \\ \x & \Xb
			\end{bmatrix}
			\begin{bmatrix}
				-1\\\u
			\end{bmatrix} = \u\u\T\bullet\Xb-2\u\T\x + 1 = 0,
			\quad
			\Leftrightarrow
			\quad 
			\begin{array}{r}
				\u\T\x = 1, \\
				\u\u\T\bullet\Xb = 1,
			\end{array}
		\end{align*}
		again by assumption and \cref{lem:EquivalentLinear}. Define $\FF_i\coloneqq \lrbr{[\x\T,y_i]\T\in\KK\times\R_+\colon \aa_i\T\x + y_i = (\d)_i}, \ i \in\irg{1}{m},$ where $\aa_i\in\R^n, \ i \in\irg{1}{m}$ are the rows of the matrix $\Ab$, and also define $\FF_0\coloneqq\lrbr{\x\in\KK \colon \u\T\x =  1}$. We see that $\FF_0 = \FF_0^{\x}\subseteq \FF_i^{\x}, \ i \in\irg{1}{m}$ by assumption, which entails iii.\ by \cref{prop:SufficientConditionForXXassumption}. The set $\FF_0$ is bounded by \cref{prop:BoundedFFxx}, so that ii.\ holds as well by the final statement in  \cref{prop:BoundednessofXX} and i.\ is also immediate since the coefficient of $y_i$ in the definition of $\FF_i$ is equal to $1\in \mathrm{int}(\redd{\R_+^*})$. The argument then proceeds as in the first part, which concludes the proof. 
	\end{proof}
	
	Testing condition a) of the above theorem is \redd{straightforward}, given a set of candidates for $\v_i,\ i \in\irg{1}{m}$, while b) seems more cumbersome due to the set-containment requirement. However, by combining b) with \cref{lem:SetImplications} we can obtain the following testable equivalent.

	\begin{prop}\label{prop:PracticalConditions}
		Let the set $\PP\coloneqq\lrbr{\x\in\KK\colon \aa_i\T\x\leq d_i, \ i\in\irg{1}{m}}$ be bounded, where $\KK\subseteq \R^n$ is a \red{ground} cone that is not a line.
		For a matrix $\Mb\in\SS^{n+1}$ and a vector $\v\coloneqq [-1,\u]\in\R\times\KK$, following two statements are equivalent:
		\begin{enumerate}
			\item[a)] The kernel of $\Mb$ contains $\v$, $\u\in\mathrm{int}(\KK^*),$ and $\lrbr{\x\in\KK\colon \u\T\x =1}\subseteq \PP$. 
			\item[b)] $\Mb\v=\oo$ and there is an $\ggl\in\R_+^m$ such that $(\ggl)_i\leq d_i,\ (\ggl)_i\u-\aa_i\in\KK^*, \ i \in\irg{1}{m}$. 				
		\end{enumerate}		
	\end{prop}
	
	\begin{proof}
		The fact that b) implies a) stems, firstly, from the definition of a kernel, while the set-containment is a consequence of \cref{lem:SetImplications} with $\KK_y \coloneqq \R_+$, analogous to the argument in \cref{prop:CharacterizationOfsufficientConditionOnXXi}. Finally, $\u\in\mathrm{int}(\KK^*)$ follows from \cref{prop:BoundedFFxx}, where $(\x,\u)$ takes the role of $(\y,\g)$ and  $\lrbr{\oo}\redd{\subseteq\R^n}$ takes the role of $\KK_0$, while $\lrbr{\x\in\KK\colon \u\T\x =1}$ takes the role of $\FF(\bar\x)$ (whose important property is that the righthand side of the defining equation is nonnegative, which is the case for the present instance), and is bounded since it is a subset of $\PP$, and further $\KK$ is closed convex but not a line, as required. 
		For the converse we can take the respective converse implications of the results cited before, where particularly for the application of \cref{lem:SetImplications} in that direction we need to show that either condition 1) or 2) in that lemma holds, and we will argue the latter\redd{, i.e.,} that the set $\lrbr{\x\in\KK\colon \u\T\x=1}$ contains a point $\bar\x\in\mathrm{ri}(\KK)$. So assume to the contrary that $\lrbr{\x\in\R^n\colon \u\T\x = 1}\cap \mathrm{ri}(\KK) = \emptyset$, then by \cite[Theorems 11.2. and 11.7.]{rockafellar_convex_2015} (which are applicable since the relative interior of a nonempty convex set is never empty by \cite[Theorem 6.2.]{rockafellar_convex_2015}) there is a $\w\in\R^n\setminus\lrbr{\oo}$ such that 
		\begin{align}
			&\w\T\x\leq 0,\quad \forall \x\in\lrbr{\x\in\R^n\colon \u\T\x = 1},
			\label{eqn:Seperation}\\
			&\w\T\x\geq 0,\quad \forall \x\in\mathrm{ri}\KK. \nonumber
		\end{align}
		The former condition necessitates $\w=\ga\u$ (otherwise the linear equation system $\w\T\x = 1,\ \u\T\x = 1$ would necessarily have a solution contrary to \eqref{eqn:Seperation}), hence the latter condition yields $\ga\u\T\x\geq0,$ while also $\u\T\x> 0$ whenever $\x\in\mathrm{ri}(\KK)\subseteq \KK\setminus\lrbr{\oo}$ since $\u\in\mathrm{int}(\KK^*)$. This implies that $\ga\geq 0$, but $\w\neq \redd{\oo},$ hence $\ga>0$.  But then \eqref{eqn:Seperation} says that $\u\T\x = 1$ implies that $0\geq \w\T\x = \ga\u\T\x = \ga >0$ which is absurd. 
	\end{proof}
	
	The above proposition is useful in determining whether a candidate for $\u$ is appropriate, since for a fixed $\u$ the required $\ggl$ can be found by solving a feasibility problem whose tractability depends on the tractability of $\KK$. However, if the kernel of $\Mb$ has dimension greater than one, then finding a candidate for $\u$ that fulfills $\Mb\v = \oo$ and $\ggl\leq \d,\ (\ggl)_i\u-\aa_i\in\KK^*, \ i \in\irg{1}{m}$, amounts to solving a nonlinear feasibility problem, due to the bilinear interaction of $\ggl$ and $\u$ in the conic constraint. This may be a difficult task, but we will defer the investigation of viable strategies to future research. We close this discussion with a small example that illustrates the usefulness of the techniques developed in this section. 
	
	\begin{exmp}
		Consider the problem 
		\begin{align*}
			\min_{\x\in\R_+^2}
			\lrbr{
				-\x\T\x
				\colon 
				\Fb\x\leq \d					
			},\ \quad 
			\Fb \coloneqq 
			\begin{bmatrix}
				1 & 2 \\ 2 &1 
			\end{bmatrix},\ 
			\d \coloneqq 
			\begin{bmatrix}
				1 \\ 1 
			\end{bmatrix},
		\end{align*}
		which has a sparse relaxation given by
		\begin{align*}
			\min_{\x,\Xb,\z_i,y_i,\Yb}
			\lrbr{
				-\Ib\bullet\Xb 
				\colon 
				\begin{array}{rl}
					\f_i\T\x + y_i = 1,
					& i =1,2,
					\\
					\f_i\f_i\T\bullet\Xb + 2\f_i\T\z_i+Y_i = 1,
					& i  =1,2,
					\\
				\end{array}								
				\begin{bmatrix}
					1 & \x\T &y_i \\
					\x& \Xb  &\z_i\\
					y_i& \z_i\T& Y_i
				\end{bmatrix} \in \CPP(\R_+^4),\
				i  =1,2 		
			},
		\end{align*}
		where $\f_1 = [1,2]\T, \ \f_2=[2,1]\T$.  
		This relaxation can be solved exactly as a positive semidefinite optimization problem since $\CPP(\R_+^4)$ is identical to the set of entrywise nonnegative positive semidefinite matrices (see \cite{maxfield_matrix_1962}). When we tried to solve this problem using Mosek (via the YALMIP toolbox, see \cite{lofberg_yalmip_2004}), the solver would produce the solution
		\begin{align*}
			\Xb = 
			\begin{bmatrix}
				1/8 & 0 \\ 0 & 1/8
			\end{bmatrix},\
			\x  = [1/4 , 1/4]\T,\ 
			\begin{array}{l}				
				\z_1= [1/8 , 0]\T,\ Y_1 = 1/8,\ y_1 = 1/4,\\
				\z_2= [0 , 1/8]\T,\ Y_2 = 1/8,\ y_2 = 1/4,\\
			\end{array}						
		\end{align*} 
		which yields the correct optimal value of $-1/4$. Unfortunately, the corresponding matrix blocks have a rank larger than one, and the $\x$-part of that solution, while obviously feasible, produces an upper bound of $-1/8$, so neither of the two standard strategies allows us to certify that the relaxation is actually tight without knowing the optimal value beforehand. However, the vector $\u\coloneqq [2,2]\T\in\mathrm{int}(\redd{(\R^2_+)^*})$ fulfills the requirements of \cref{thm:OptimalityConditions} since 
		\begin{align*}
			\begin{bmatrix}
				1 & 1/4 & 1/4 \\ 
				1/4 & 1/8 &0 \\
				1/4 & 0   &1/8
			\end{bmatrix}
			\begin{bmatrix}
				-1\\ 2 \\ 2
			\end{bmatrix} = 
			\begin{bmatrix}
				0 \\ 0 \\ 0
			\end{bmatrix}
			\mbox{ and }
			\gl_1
			\begin{bmatrix}
				2 \\ 2 
			\end{bmatrix}
			-
			\begin{bmatrix}
				1 \\ 2 
			\end{bmatrix}\geq 0, \
			\gl_2
			\begin{bmatrix}
				2 \\ 2 
			\end{bmatrix}
			-
			\begin{bmatrix}
				2 \\ 1 
			\end{bmatrix}\geq 0, 
		\end{align*}
		whenever $\gl_i= 1\leq (\d)_i\geq 0, \ i=1,2,$ which in conjunction with \cref{prop:PracticalConditions} shows that the desired completely positive completion of the two matrix blocks must exist. Indeed, direct calculation shows that it is given by 
		\begin{align*}
			\begin{bmatrix}
				1 & 1/4 & 1/4 & 1/4 & 1/4 \\ 
				1/4 & 1/8 &0  & 1/8 & 0   \\
				1/4 & 0   &1/8& 0   & 1/8 \\ 
				1/4 & 1/8 &   & 1/8 & 0   \\ 
				1/4 & 0   &1/8& 0   & 1/8 \\ 
			\end{bmatrix} =
			1/2 
			\begin{bmatrix}
				1 \\ 0 \\ 1/2 \\ 0 \\ 1/2
			\end{bmatrix}			
			\begin{bmatrix}
				1 \\ 0 \\ 1/2 \\ 0 \\ 1/2
			\end{bmatrix}\T	
			+
			1/2
			\begin{bmatrix}
				1 \\ 1/2\\  0 \\ 1/2 \\ 0
			\end{bmatrix}
			\begin{bmatrix}
				1 \\ 1/2\\  0 \\ 1/2 \\ 0
			\end{bmatrix}\T
			\in\CPP(\R_+^5),						
		\end{align*}
		which would be \redd{a corresponding} optimal solution for the full, completely positive reformulation of the original QP. Note that the two vectors that make up the two rank-one components are both optimal solutions of that QP as well. This demonstrates that our optimality condition is valuable if the QP has multiple solutions and the sparse relaxation happens to find a solution that can be completed to the convex combination of rank-one solutions of the full reformulation.  
	\end{exmp}
	
	\section*{Conclusion and outlook}
	In this text, we derived a new set-completely positive matrix completion result, which gives sufficient conditions for a partial matrix with an arrowhead specification pattern and width equal to one to be completable with respect to certain set-completely positive matrix cones, including the classical completely positive cone. While usually such completion results are used to close sparse relaxation gaps, we went the opposite direction and used an exactness result in order to infer our completion result. We showed how this completion result can be applied to an instance of a sparse reformulation that, in general, does not lie within the regime of our exactness result. 
	
	The caveat is that our techniques require the sparse completely positive reformulation to be solved exactly, i.e.\ without replacing the $\CPP$-constraint with one of its many approximations, which is possible for example via the techniques in \cite{bundfuss_adaptive_2009,gabl_solving_2023}. However, approximations of the complicated conic constraints are also popular, and at this point, it is not clear how these relaxations behave when applied to the sparse completely positive relaxation vs. the full completely positive reformulation. It stands to reason that any approximation scheme incurs a smaller loss if it is applied to the smaller $\CPP$ constraints of the sparse reformulation, so that in case the sparse relaxation is tight, the respective approximation performs at least as good as if the same approximation is applied to the full reformulation. However, it seems difficult to quantify these losses caused by approximations in relation to the order of the matrix block. But we hope that we can investigate this relation in future research. 
	
	\paragraph{Acknowledgements:}
	M. G. was supported by the Austrian Science Fund (FWF, project W 1260-N35) as a Ph.D.\ student member of the Vienna Graduate School on Computational Optimization (VGSCO). He is currently funded by the FWF project ESP 486-N as a post-doc. The author is indebted to the editorial team and the two anonymous reviewers whose diligent feedback helped greatly improve the content of this text.
	
	\bibliography{Literature}

\begin{thebibliography}{10}

\bibitem{anstreicher_convex_2012}
K.~M. Anstreicher.
\newblock On convex relaxations for quadratically constrained quadratic
  programming.
\newblock {\em Mathematical programming}, 136(2):233--251, 2012.

\bibitem{berman_completely_2003}
A.~Berman and N.~Shaked-Monderer.
\newblock {\em Completely positive matrices}.
\newblock World Scientific, 2003.

\bibitem{bomze_copositive_2000}
I.~Bomze, M.~Dür, E.~De~Klerk, C.~Roos, A.~Quist, and T.~Terlaky.
\newblock On copositive programming and standard quadratic optimization
  problems.
\newblock {\em Journal of Global Optimization}, 18(4):301--320, 2000.

\bibitem{bomze_optimization_2023}
I.~M. Bomze and M.~Gabl.
\newblock Optimization under uncertainty and risk: {Quadratic} and copositive
  approaches.
\newblock {\em European Journal of Operational Research}, 310(2):449--476, Oct.
  2023.

\bibitem{bomze_two-stage_2022}
I.~M. Bomze, M.~Gabl, F.~Maggioni, and G.~C. Pflug.
\newblock Two-stage {Stochastic} {Standard} {Quadratic} {Optimization}.
\newblock {\em European Journal of Operational Research}, 299(1):21--34, 2022.

\bibitem{bomze_think_2012}
I.~M. Bomze, W.~Schachinger, and G.~Uchida.
\newblock Think co(mpletely)positive! {Matrix} properties, examples and a
  clustered bibliography on copositive optimization.
\newblock {\em Journal of Global Optimization}, 52(3):423--445, 2012.

\bibitem{bundfuss_adaptive_2009}
S.~Bundfuss and M.~Dür.
\newblock An adaptive linear approximation algorithm for copositive programs.
\newblock {\em SIAM Journal on Optimization}, 20(1):30--53, 2009.

\bibitem{burer_copositive_2009}
S.~Burer.
\newblock On the copositive representation of binary and continuous nonconvex
  quadratic programs.
\newblock {\em Mathematical Programming}, 120(2):479--495, 2009.

\bibitem{burer_copositive_2012}
S.~Burer.
\newblock Copositive programming.
\newblock In {\em Handbook on semidefinite, conic and polynomial optimization},
  pages 201--218. Springer, 2012.

\bibitem{dickinson_copositive_2013}
P.~J.~C. Dickinson.
\newblock {\em The copositive cone, the completely positive cone and their
  generalisations}.
\newblock {PhD} {Thesis}, Groningen: s.n., 2013.

\bibitem{drew_completely_1998}
J.~H. Drew and C.~R. Johnson.
\newblock The completely positive and doubly nonnegative completion problems.
\newblock {\em Linear and Multilinear Algebra}, 44(1):85--92, 1998.

\bibitem{dur_copositive_2010}
M.~Dür.
\newblock Copositive programming – a survey.
\newblock In M.~Diehl, F.~Glineur, E.~Jarlebring, and W.~Michiels, editors,
  {\em Recent advances in optimization and its applications in engineering},
  pages 3--20. Springer, Berlin Heidelberg, 2010.

\bibitem{eichfelder_set-semidefinite_2013}
G.~Eichfelder and J.~Povh.
\newblock On the set-semidefinite representation of nonconvex quadratic
  programs over arbitrary feasible sets.
\newblock {\em Optimization Letters}, 7(6):1373--1386, 2013.

\bibitem{fukuda_exploiting_2001}
M.~Fukuda, M.~Kojima, K.~Murota, and K.~Nakata.
\newblock Exploiting sparsity in semidefinite programming via matrix completion
  {I}: {General} framework.
\newblock {\em SIAM Journal on Optimization}, 11(3):647--674, 2001.

\bibitem{gabl_sparse_2023}
M.~Gabl.
\newblock Sparse conic reformulation of structured {QCQPs} based on copositive
  optimization with applications in stochastic optimization.
\newblock {\em Journal of Global Optimization}, 87(1):221--254, Sept. 2023.

\bibitem{gabl_solving_2023}
M.~Gabl and K.~Anstreicher.
\newblock Solving {Nonconvex} {Optimization} {Problems} using {Outer}
  {Approximations} of the {Set}-{Copositive} {Cone}.
\newblock {\em Optimization Online}, 2023.

\bibitem{gabl_finding_2024}
M.~Gabl and I.~M. Bomze.
\newblock Finding quadratic underestimators for optimal value functions of
  nonconvex all-quadratic problems via copositive optimization.
\newblock {\em EURO Journal on Computational Optimization}, 12:100100, 2024.

\bibitem{goemans_improved_1995}
M.~X. Goemans and D.~P. Williamson.
\newblock Improved approximation algorithms for maximum cut and satisfiability
  problems using semidefinite programming.
\newblock {\em Journal of the ACM (JACM)}, 42(6):1115--1145, 1995.

\bibitem{grone_positive_1984}
R.~Grone, C.~R. Johnson, E.~M. Sá, and H.~Wolkowicz.
\newblock Positive definite completions of partial {Hermitian} matrices.
\newblock {\em Linear Algebra and its Applications}, 58:109--124, Apr. 1984.

\bibitem{kim_doubly_2020}
S.~Kim, M.~Kojima, and K.-C. Toh.
\newblock Doubly nonnegative relaxations are equivalent to completely positive
  reformulations of quadratic optimization problems with block-clique graph
  structures.
\newblock {\em Journal of Global Optimization}, 77:513--541, 2020.

\bibitem{lofberg_yalmip_2004}
J.~Löfberg.
\newblock {YALMIP} : {A} {Toolbox} for {Modeling} and {Optimization} in
  {MATLAB}.
\newblock In {\em In {Proceedings} of the {CACSD} {Conference}}, Taipei,
  Taiwan, 2004.

\bibitem{maxfield_matrix_1962}
J.~E. Maxfield and H.~Minc.
\newblock On the {Matrix} {Equation} {X}'{X} = {A}.
\newblock {\em Proceedings of the Edinburgh Mathematical Society},
  13(2):125--129, 1962.
\newblock Edition: 2009/01/20 Publisher: Cambridge University Press.

\bibitem{motzkin_copositive_1818}
T.~Motzkin.
\newblock Copositive quadratic forms.
\newblock {\em National Bureau of Standards Report}, 1952:11--22, 1818.

\bibitem{rockafellar_convex_2015}
R.~T. Rockafellar.
\newblock {\em Convex analysis}.
\newblock Princeton university press, 2015.

\bibitem{shor_quadratic_1987}
N.~Z. Shor.
\newblock Quadratic optimization problems.
\newblock {\em Soviet Journal of Computer and Systems Sciences}, 25(6):1--11,
  1987.

\bibitem{sturm_cones_2003}
J.~F. Sturm and S.~Zhang.
\newblock On cones of nonnegative quadratic functions.
\newblock {\em Mathematics of Operations Research}, 28(2):246--267, 2003.

\bibitem{vandenberghe_chordal_2015}
L.~Vandenberghe and M.~S. Andersen.
\newblock Chordal graphs and semidefinite optimization.
\newblock {\em Foundations and Trends in Optimization}, 1(4):241--433, 2015.

\bibitem{vandenberghe_semidefinite_1996}
L.~Vandenberghe and S.~Boyd.
\newblock Semidefinite {Programming}.
\newblock {\em SIAM Review}, 38(1):49--95, 1996.
\newblock Publisher: Society for Industrial and Applied Mathematics.

\bibitem{yakubovich_s-procedure_1971}
V.~A. Yakubovich.
\newblock S-procedure in nonlinear control theory.
\newblock {\em Vestnik Leningrad University}, 1:62--77, 1971.

\end{thebibliography}
	\bibliographystyle{abbrv}	
\end{document}